\documentclass[11pt]{article}

\usepackage[a4paper, total={5.8in, 8in}]{geometry}

\usepackage{tikz}
\usepackage{microtype}
\usepackage{lipsum}
\usepackage{graphicx}
\usepackage{float}
\usepackage{biblatex}
\usepackage[hidelinks]{hyperref}
\usepackage{xurl}
\usepackage{mathtools}
\usepackage{amssymb}
\usepackage{amsthm}
\usepackage{amsmath}
\usepackage{setspace}
\usepackage[utf8]{inputenc}
\usepackage[english]{babel}
\usepackage{framed}
\usepackage{caption}
\usepackage[hang,flushmargin]{footmisc}
\usepackage{xcolor}
\usepackage{musicography}
\usepackage{csquotes}

\hypersetup{colorlinks,linkcolor={blue},citecolor={green},urlcolor={black}}  
\hypersetup{breaklinks=true}

\usepackage{pgfplots}
\usetikzlibrary{patterns, arrows.meta}

\newtheorem{theorem}{Theorem}
\newtheorem{conjecture}{Conjecture}

\newtheorem{lemma}{Lemma}

\newtheorem{proposition}{Proposition}
\newtheorem{remark}{Remark}

\makeatletter

\newcommand\ackname{Acknowledgements}
\if@titlepage
  \newenvironment{acknowledgements}{%
      \titlepage
      \null\vfil
      \@beginparpenalty\@lowpenalty
      \begin{center}%
        \bfseries \ackname
        \@endparpenalty\@M
      \end{center}}%
     {\par\vfil\null\endtitlepage}
\else

\fi

\title{On the Donaldson-Scaduto conjecture}
\author{Saman Habibi Esfahani, Yang Li}
\date{\today}

\begin{document}
\maketitle

\begin{abstract}
Motivated by $G_2$-manifolds with coassociative fibrations in the adiabatic limit, Donaldson and Scaduto conjectured the existence of associative submanifolds homeomorphic to a three-holed $3$-sphere with three asymptotically cylindrical ends in the $G_2$-manifold $X \times \mathbb{R}^3$, or equivalently similar special Lagrangians in the Calabi-Yau 3-fold $X \times \mathbb{C}$, where $X$ is an $A_2$-type ALE hyperkähler 4-manifold. We prove this conjecture by solving a real Monge-Amp\`ere equation with a singular right-hand side, which produces a potentially singular special Lagrangian. Then, we prove the smoothness and asymptotic properties for the special Lagrangian using inputs from geometric measure theory. 
The method produces many other asymptotically cylindrical $U(1)$-invariant special Lagrangians in $X\times \mathbb{C}$, where $X$ arises from the Gibbons-Hawking construction.
\end{abstract}

\section{Introduction}

Simon Donaldson \cite{MR3702382} initiated a program to study $G_2$-manifolds through coassociative K3 fibrations $\pi:M^7\to B^3$ over a 3-dimensional base $B^3$, in the adiabatic limit where the diameters of the K3 fibers shrink to zero. This program is expected to lead to large classes of new examples of compact torsion-free $G_2$-manifolds. 

Subsequent work of Donaldson and Scaduto \cite{MR4479718}  provided a conjectural limiting description of certain associative submanifolds in the adiabatic setting. Roughly speaking, the generic part of the associative submanifolds is fibered over one-dimensional gradient flowlines inside the base $B^3$. These flowlines are allowed to end on the discriminant locus of $\pi: M^7\to B^3$, and generically, three flowlines can meet to form a triple junction point. Donaldson and Scaduto made several conjectures concerning the existence of the local model for the triple junction.

In the ``global'' version of the conjecture, let $(X^4, I_1, I_2, I_3)$ be a hyperkähler K3 surface, and let $\alpha_1, \alpha_2,\alpha_3$ be $(-2)$-classes in $H_2(X^4; \mathbb{R})$, namely $\alpha_i^2 = -2$ with respect to the intersection product, such that $\alpha_1 + \alpha_2 + \alpha_3 = 0$. Each $\alpha_i$ determines a complex structure $J_i$ in the $S^2$-family of complex structures, so that for suitable choices of $v_i\in \mathbb{R}^3$, 
we have three cylindrical associative submanifolds $P_i := \Sigma_i \times (\mathbb{R}^+ v_i) \subset X^4 \times \mathbb{R}^3$ for $i\in\{1,2,3\}$.

\begin{conjecture}[Donaldson-Scaduto]\label{Donaldson-Scaduto-conjecture}
 There is an associative submanifold homeomorphic to a three-holed $3$-sphere $P \subset X^4 \times \mathbb{R}^3$ with three ends asymptotic to cylinders $P_1, P_2$, and $P_3$.
\end{conjecture}

In the ``local'' version of the conjecture, the K3 surface is replaced with an $A_2$-type ALE gravitational instanton $X^4_{A_2}$. In the Gibbons-Hawking ansatz, $X^4_{A_2}$ is defined as the completion of a $U(1)$-bundle over $\mathbb{R}^3 \setminus \{p_1, p_2, p_3\}$. We make the genericity assumption that $p_1, p_2, p_3$ are three non-collinear points in $\mathbb{R}^3$, which corresponds to the condition that the three $(-2)$-spheres $\Sigma_i$ are holomorphic with respect to different complex structures. The relation to the global version is that when the K3 surface is the small desingularization (e.g., smoothing or resolution) of an orbifold with local $A_2$-singularity, then the local version captures the metric behavior near the desingularization region.

We prove the local version in this note.

\begin{theorem}[Donaldson-Scaduto conjecture, local version]\label{Donaldson-Scaduto-conjecture} There exists a $U(1)$-invariant associative submanifold  $P \subset X^4_{A_2} \times \mathbb{R}^3$  homeomorphic to a three-holed $3$-sphere, with three ends asymptotic to the half-cylinders $\Sigma_i \times (\mathbb{R}^+ v_i)$, where $i \in \{1,2,3\}$.
\end{theorem}
\begin{figure}[H]
\includegraphics[width=12cm]{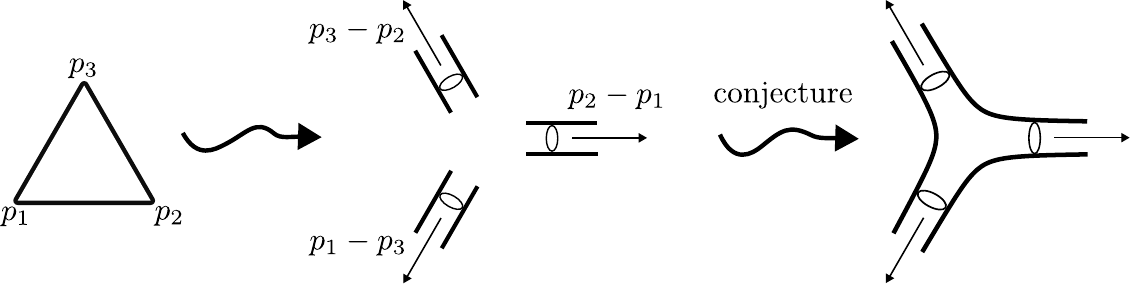}
\centering
  \caption{Donaldson-Scaduto conjecture.}
  \label{Donaldson-Scaduto-figure}
\end{figure}
In fact, since the vectors $v_1, v_2$, and $v_3$ lie in a plane, say $\mathbb{R}^2 \times \{0\} \subset \mathbb{R}^3$, the associative submanifold $P$ can be equivalently interpreted as a special Lagrangian submanifold in $X^4\times \mathbb{R}^2$ with an appropriate Calabi-Yau structure. Our method readily generalizes to the $A_{n-1}$-type ALE or ALF gravitational instantons $X$, where $n\geq 3$, and the monopole points $p_1,\ldots, p_n$ in the Gibbons-Hawking ansatz are in the ``convex position'', namely they are the vertices of a convex polygon in a plane in $\mathbb{R}^3$, arranged in the counterclockwise orientation. We equip $X\times \mathbb{R}^2$ with a natural product Calabi-Yau 3-fold structure.

\begin{theorem}[Generalization] \label{generalizationDonaldsonScaduto}  There is an $(n-1)$-dimensional family of $U(1)$-invariant special Lagrangian submanifolds in the Calabi-Yau 3-fold $X\times \mathbb{R}^2$, each homeomorphic to an $n$-holed 3-sphere and with $n$ asymptotically cylindrical ends, modeled on the product of $\Sigma_i\subset X$ and $\{ (p_{i+1}-p_i) \cdot y= c_i   \} \subset \mathbb{R}^2$, 
 where the parameters $\{c_i\}_{i=1}^n$ satisfy one constraint, $\sum_{i=1}^n c_i=0$.
\end{theorem}

The $(n-1)$ parameters geometrically correspond to the translation of the $n$ asymptotic cylinder ends, subject to one constraint coming from the vanishing of the integral of $\text{Im}(\Omega)$. Specifically, two of these parameters account for global translations along the $\mathbb{R}^2$ direction, while the remaining $(n-3)$ parameters yield geometrically distinct special Lagrangians. Moreover, these special Lagrangians remain rigid after fixing the asymptotic conditions, as studied in \cite[Theorem 43.]{MR4495257}.

\begin{remark}
\emph{
In \cite[Theorem 1]{MR4945554}, it is shown that any special Lagrangian submanifold homeomorphic to an $n$-holed 3-sphere with $n$ asymptotically cylindrical ends, satisfying the conditions of Theorem \ref{generalizationDonaldsonScaduto}, belongs to the $(n-1)$-dimensional family of $U(1)$-invariant special Lagrangians constructed in this paper.}
\end{remark}

We expect these special Lagrangians to be useful as building blocks in the gluing construction of new special Lagrangians in ``local Calabi-Yau 3-folds'' admitting a fibration of $A_{n-1}$-type spaces over a Riemann surface. 

\vspace{5pt}

\textbf{Plan of the paper.} 
We focus on Theorem \ref{generalizationDonaldsonScaduto}. In Section \ref{Preliminaries}, we introduce the geometric structures on the ambient spaces. In Section \ref{setup}, we dimensionally reduce the $U(1)$-invariant special Lagrangian conditions to an equation for surfaces in the symplectic quotient. Under an additional graphical assumption, this leads to a 2-dimensional real Monge-Amp\`ere equation for some potential $\varphi$ over the convex polygon with vertices $p_1,\ldots, p_n$. In Section \ref{RMAequation}, we solve the appropriate Dirichlet problem for $\varphi$, where the boundary data is given by affine linear functions on each edge of the polygon. The $U(1)$-bundle $L^{\circ}$ over the gradient graph of the solution over the open solid convex polygon is an open $U(1)$-invariant special Lagrangian in $X \times \mathbb{R}^2$. 

In Section \ref{tangentcone}, we take the closure of $L^{\circ}$, denoted by $L$. To prove $L$ is a closed submanifold, we need two extra ingredients. First, the gradient of $\varphi$ diverges to infinity near the edges of the convex polygon away from the vertices, and therefore, the edges give no contribution to the closure of the special Lagrangian; this uses some analysis on the real Monge-Amp\`ere equation. Second, the vertex contributions introduce only smooth points to the special Lagrangian; this uses some geometric measure theory, as well as the classification of $U(1)$-invariant special Lagrangian cones in $\mathbb{C}^3$, following some earlier idea of Joyce \cite{MR2122282}. In Section \ref{Asymptotic-behavior}, we prove an exponential decay estimate and show that $L$ has the expected asymptotic behavior. In Section \ref{topology}, we prove $L$ is homeomorphic to an $n$-holed 3-sphere. This concludes the proof of Theorem \ref{generalizationDonaldsonScaduto}. 

\vspace{5pt} 

\textbf{Acknowledgement.} We thank our common teacher Prof. Simon Donaldson for suggesting this problem to us and the referee for useful comments on the previous draft. S.H. thanks Mark Haskins and Rafe Mazzeo for discussions on related topics. Y.L. was supported by the Clay Maths Research Fellowship. 

\section{Preliminaries: ambient spaces}\label{Preliminaries}

We recall the hyperkähler structure on the $U(1)$-invariant gravitational instanton $X$, and describe the Calabi-Yau structure on $Z = X \times \mathbb{R}^2$ and the $G_2$-structure on $M = X \times \mathbb{R}^3$.
\vspace{5pt}

\noindent
\textbf{Hyperkähler structure.} 
Let $X$ be a complete non-compact $U(1)$-invariant hyperkähler 4-manifold, given by the Gibbons-Hawking construction as follows. Let $n\geq 3$, and let $p_1, p_2 \ldots, p_n$ be $n$ distinct points in $\mathbb{R}^3$. We will assume that $p_i$ are contained in the plane $\mathbb{R}^2\times \{0\} \subset \mathbb{R}^3$, and in the `convex position', namely they are the vertices of a convex polygon, arranged in counterclockwise order. In the $n=3$ case, up to coordinate rotation, this amounts to the genericity assumption that $p_1,p_2,p_3$ are non-colinear. 
 
Let $u_1,u_2,u_3$ denote the coordinates on $\mathbb{R}^3$. Let 
\begin{align*}
\pi: X^{\circ} \to \mathbb{R}^3 \setminus \{p_1, p_2, \ldots, p_n\},    
\end{align*}
be a principal $U(1)$-bundle, with Chern class 1 on small $S^2$ around each point $p_i$. Let $V: \mathbb{R}^3 \setminus \{p_1, p_2, \ldots, p_n\} \to \mathbb{R}$ be the positive harmonic function 
\begin{align*}
    V(u) =  A + \sum_{i=1}^n \frac{1}{2 |u-p_i|},\quad A=\text{constant}\geq 0,
\end{align*}
and let $\theta$ be a $U(1)$-connection on $X^{\circ}$ with curvature 2-form $*dV$. The Gibbons-Hawking ansatz describes a hyperkähler structure on $X^{\circ}$ given by the symplectic forms
\begin{align*}
\omega_1 = \theta\wedge du_1  + V du_{2} \wedge du_{3}, \quad 
\omega_2 = \theta\wedge du_2   + V du_{3} \wedge du_{1}, \quad 
\omega_3 =\theta\wedge du_3   + V du_{1} \wedge du_{2},
\end{align*}
and the metric $g = V^{-1} \theta^2 + V \sum_{i=1}^3 du_i^2$. 

The coordinates $u_1,u_2, u_3$ are the moment maps with respect to the symplectic forms $\omega_1, \omega_2, \omega_3$, respectively. The manifold $X$ is obtained by adding a point above each $p_i$, and the hyperkähler structure extends smoothly to $X$, with the corresponding complex structures $I_1, I_2, I_3$.  In fact, for each $(a_1,a_2,a_3)\in S^2 \subset \mathbb{R}^3$, we obtain a complex structure $\sum_{i=1}^3 a_i I_i$ on $X$. For $A=0$, the hyperkähler manifold $X$ is an $A_{n-1}$-type ALE space, and for $A>0$, it is an $A_{n-1}$-type ALF space.

Let $\Sigma_i = \pi^{-1}[p_{i},p_{i+1}]$ be the preimage of $[p_{i},p_{i+1}]$, the line segment from $p_{i}$ to $p_{i+1}$, and for convenience denote $p_{n+1}=p_1$. Each $\Sigma_i$ is a 2-sphere, which is holomorphic with respect to the complex structure associated with the vector $v_i$, 
\begin{align*}
    v_{i} = \frac{(p_{i+1} - p_{i})}{|p_{i+1} - p_{i}|} \in S^2 \cap (\mathbb{R}_{(u_1,u_2)}^2 \times \{0\}) \subset \mathbb{R}_{(u_1,u_2,u_3)}^3.
\end{align*}

\vspace{5pt}

\noindent
\textbf{Calabi-Yau structure.} Let $Z = X \times \mathbb{R}^2_{(y_1,y_2)}$. The $6$-dimensional manifold $Z$ can be equipped with the Calabi-Yau structure 
\[
g_Z = g_X + g_{\mathbb{R}^2}, \quad \omega = \omega_3 + dy_2 \wedge dy_1,\quad \Omega = (\omega_1 + i \omega_2) \wedge (dy_2 + i dy_1).
\]
where $y_1, y_2$ denote the coordinates on $\mathbb{R}^2$. Note that with our convention
\begin{align*}
    \omega^3 = \frac{3\sqrt{-1}}{4} \Omega \wedge \overline{\Omega}.
\end{align*}
We extend the $U(1)$-action on $X$ to a $U(1)$-action on $Z$ by $e^{i t} \cdot (q,y) \to (e^{it} \cdot q, y)$, for any $q \in X$ and $y \in \mathbb{R}^2$.

Let $\tilde{v}_i= Rv_i$, where $R: \mathbb{R}^2_{(u_1,u_2)}\to \mathbb{R}^2_{(y_1,y_2)}$ is the linear transformation given by the 90-degree rotation,
\[
R( \partial_{u_1})= -\partial_{y_2},\quad R( \partial_{u_2})= \partial_{y_1}.
\]
Let $L_{i} $ be $ \Sigma_{i} \times (\mathbb{R}^+ \cdot \tilde{v}_{i}) \subset X \times \mathbb{R}^2_{(y_1,y_2)}$ translated along some vector in the $\mathbb{R}^2_{(y_1,y_2)}$, so that $L_i$ is contained inside
\[
\Sigma_i \times \{ y \in \mathbb{R}^2 \; | \; (p_{i+1}-p_i)\cdot y=c_i \}\subset X\times \mathbb{R}^2,\quad i \in \{1,2,\ldots, n\}.
\]
A direct computation shows that $L_i$ are $U(1)$-invariant special Lagrangians in $Z$, 
\begin{align*}
\omega_{|_{L_i}} \equiv 0 \quad \text{ and } \quad 
\text{Im}(\Omega)_{|_{L_i}} \equiv 0, \quad \quad \text{ for } \quad \quad i \in \{1,2,\ldots, n\}.
\end{align*}

\begin{remark}
The 90-degree rotation is an artefact of our choice of holomorphic volume form $\Omega$. This choice will be useful when we later derive the real Monge-Amp\`ere equation. 
\end{remark}

\vspace{5pt}

\noindent
\textbf{$G_2$ structure.} Let $M = X \times \mathbb{R}^3$. The 7-dimensional manifold $M$ can be equipped with a torsion-free $G_2$-structure 
\begin{align*}
    \phi = dy_1 \wedge dy_2 \wedge dy_3 
      + & dy_1 \wedge du_1 \wedge \theta - V dy_1 \wedge du_2 \wedge du_3 \\
    +& dy_2 \wedge du_2 \wedge \theta 
    - V dy_2 \wedge du_3 \wedge du_1 \\
    +& dy_3 \wedge du_3 \wedge \theta
    - V dy_3 \wedge du_1 \wedge du_2,
\end{align*}
where $y_1, y_2, $ and $y_3$ denote the coordinates on $\mathbb{R}^3$. The associated $G_2$-metric is given by $g_M = g_X + g_{\mathbb{R}^3}$.

Let $P_{i} = \Sigma_{i} \times (\mathbb{R}^+ \cdot v_{i}) \subset X \times \mathbb{R}^3$, for $i \in \{1,2,3\}$, where we regard $v_i$ as a vector in $\mathbb{R}^3_{(y_1,y_2,y_3)} \cong \mathbb{R}^3_{(u_1,u_2,u_3)}$. The cylinders $P_i$ are $U(1)$-invariant associatives in $M$, namely $\phi_{|_{P_i}} = vol_{P_i}$, where the volume form $vol_{P_i}$ is defined with respect to the restriction of the Riemannian metric $g_M$ to $P_i$.

Identifying $M$ as $Z \times \mathbb{R}_{y_3}$, we have $\phi = -dy_3 \wedge \omega - \text{Im}(\Omega)$. For any submanifold $L$ in $X \times \mathbb{R}^2$, we define $\tilde{L} = \{(q,y) \in X \times \mathbb{R}^2 \; | \; (q, R^{-1}y) \in L \}$. For any $a \in \mathbb{R}$, the submanifold $P = L \times \{a\}$ is an associative submanifold of $M$ if and only if $\tilde{L}$ is a special Lagrangian submanifold in $X \times \mathbb{R}^2$. The upshot is that Theorem \ref{Donaldson-Scaduto-conjecture} follows from Theorem \ref{generalizationDonaldsonScaduto}.

\section{Dimensional reduction and real Monge-Amp\`ere equation}

We look for the dimensional reduction on the desired special Lagrangian submanifolds $L\subset X\times \mathbb{R}^2$ with cylindrical ends $L_i$ in the symplectic quotient, where the case $n=3$ has been conjectured by Donaldson and Scaduto. Under some heuristic assumptions, this leads to a certain real Monge-Amp\`ere equation with some specific Dirichlet boundary condition. The rest of this section will rigorously establish the existence and the properties of the solution. 

\subsection{The setup}\label{setup}

The $U(1)$-invariance of the asymptotic ends $L_i$ motivates us to search for $L$ among $U(1)$-invariant submanifolds. The $U(1)$-moment map associated to the symplectic form $\omega$ on $Z$ is $u_3$, which must be constant on the $U(1)$-invariant Lagrangian $L$, and the constant is fixed to be zero, since it is zero on the asymptotic cylinders. The special Lagrangian conditions $\omega_{|_{L}} \equiv 0$ and $\text{Im}(\Omega)_{|_{L}} \equiv 0$ for the $U(1)$-invariant  $L \subset X \times \mathbb{R}^2$ reduce to 
\begin{align}\label{dim-red-special-lag}
    V du_1 \wedge du_2 - dy_1 \wedge dy_2 = 0, \quad \quad 
    \text{ and } \quad \quad 
    du_1 \wedge dy_1 + du_2 \wedge dy_2 = 0,
\end{align}
respectively. The dimensionally reduced Lagrangian is
\begin{align*}
    L_{\text{red}} := L/U(1) \subset Z_{\text{red}} := u_3^{-1}(0)/U(1).
\end{align*}   
Topologically, $Z_{\text{red}}$ can be identified with $\mathbb{R}^4$ with coordinates $(u_1, u_2, y_1,y_2)$, but equipped with a degenerate reduced Kähler structure. Let 
\begin{align*}
    \pi_1: Z_{\text{red}}\to \mathbb{R}^2_{(u_1,u_2)}, \quad \quad \pi_2: Z_{\text{red}} \to \mathbb{R}^2_{(y_1,y_2)},
\end{align*}
be the natural projection maps.

The $U(1)$-reduction of the cylindrical special Lagrangians $L_i$ results in half-strips $L_{i,red}$ contained inside the cylinder
\[
   [p_{i},p_{i+1}] \times 
   \{ (p_{i+1}-p_i)\cdot y=c_i  \}
    \subset \mathbb{R}^2_{(u_1,u_2)} \times  \mathbb{R}^2_{(y_1,y_2)}, 
\]
where $[p_{i},p_{i+1}]$ denotes the closed line segment in $\mathbb{R}^2 \times \{0\} \subset \mathbb{R}^3$ that connects the points $p_{i}$ and $p_{i+1}$, and we will require the $n$ parameters $c_i$ to sum to zero (See the Appendix). Therefore, 
\begin{align*}
    \pi_1(L_{i,\text{red}}) = [p_{i},p_{i+1}] \subset \mathbb{R}^2_{(u_1,u_2)} \; \text{ and } \; 
    \pi_2(L_{i,\text{red}}) =\text{a translation of } (\mathbb{R}^+ \cdot \tilde{v}_i) \subset \mathbb{R}^2_{(y_1,y_2)},
\end{align*}
as shown in Figure \ref{strip}.
\begin{figure}[H] 
\includegraphics[width=4.6cm]{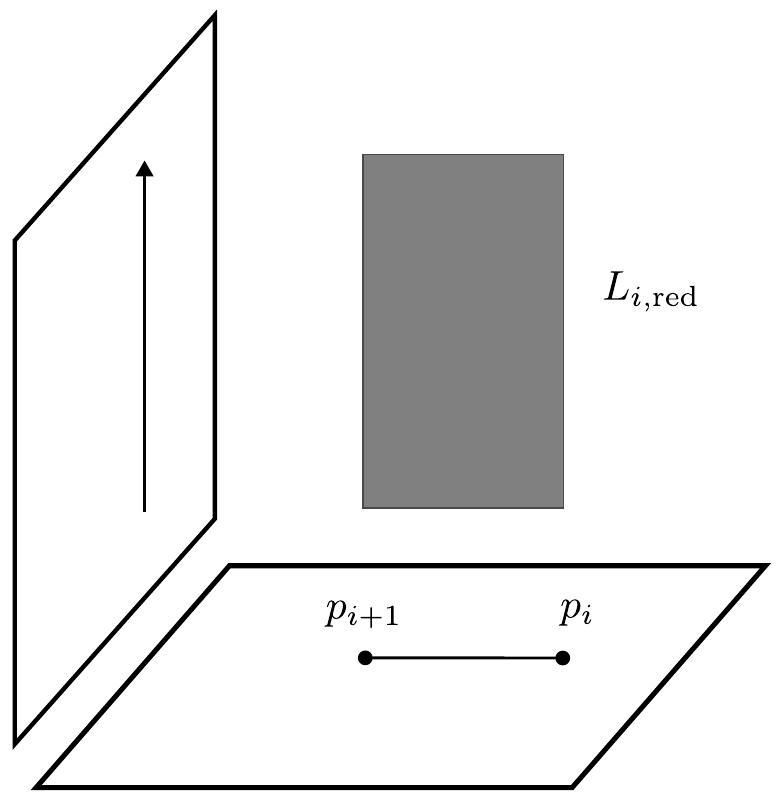}
\centering
  \caption{$L_{i,\text{red}} \subset \mathbb{R}^4 = \mathbb{R}^2_{(u_1,u_2)} \times \mathbb{R}^2_{(y_1,y_2)}$.}
   \label{strip}
\end{figure}
Therefore, $\pi_1(\cup_i L_{i,\text{red}} )$ is the 
boundary of the convex polygon 
with vertices $p_1, \ldots, p_n$, and $\pi_2( \cup_i L_{i,\text{red}} )$ is the union of $n$ rays, as shown in Figure \ref{triangle-rays2} for the case $n = 3$.
\begin{figure}[H]  
\includegraphics[width=8cm]{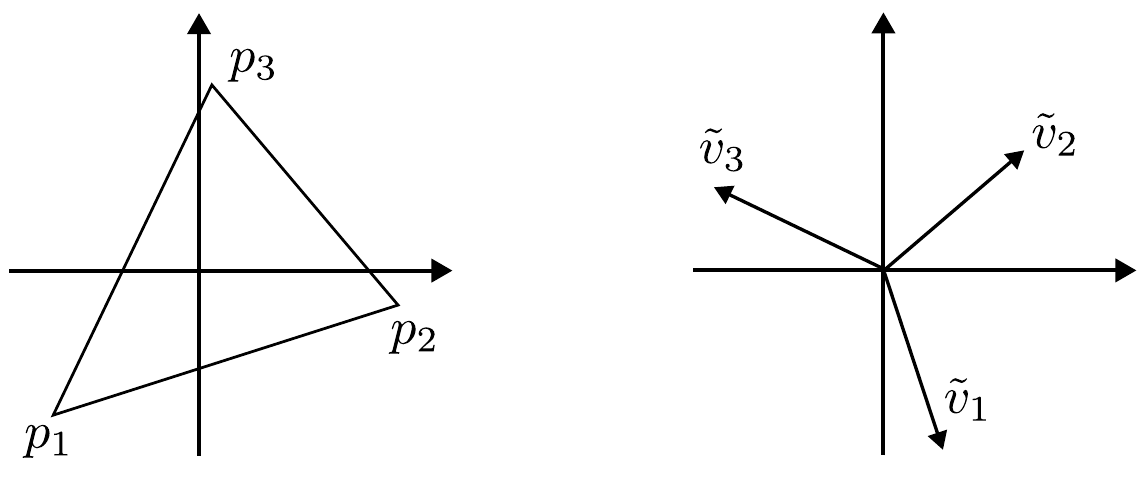}
\centering
  \caption{$\pi_1( \cup L_{i,\text{red}})$ (left) and $\pi_2(\cup L_{i,\text{red}} )$ (right).}
  \label{triangle-rays2}
\end{figure}

\textbf{Graphical case:} 
The asymptotic cylindrical requirement motivates us to look for the reduction $L_{\text{red}}$ of the conjectural special Lagrangian $L$, as (the closure of) the graph of a map
\begin{align*}
    F : U \subset \mathbb{R}^2_{(u_1,u_2)} \to \mathbb{R}^2_{(y_1, y_2)}, \quad \quad (y_1, y_2) = F(u_1, u_2),
\end{align*}
where $U$ is the interior of the convex polytope with vertices $p_1,\ldots, p_n$. 

The projection of $L_{\text{red}}$ to $\mathbb{R}^2_{(y_1,y_2)}$ is expected to be a thickening of the union of the $n$ rays, whose shape resembles an amoeba, as illustrated below in Figure \ref{DSC}. We will study these shapes more in Section \ref{soolutions-near-vertices}.
\begin{figure}[H]  
\includegraphics[width=9cm]{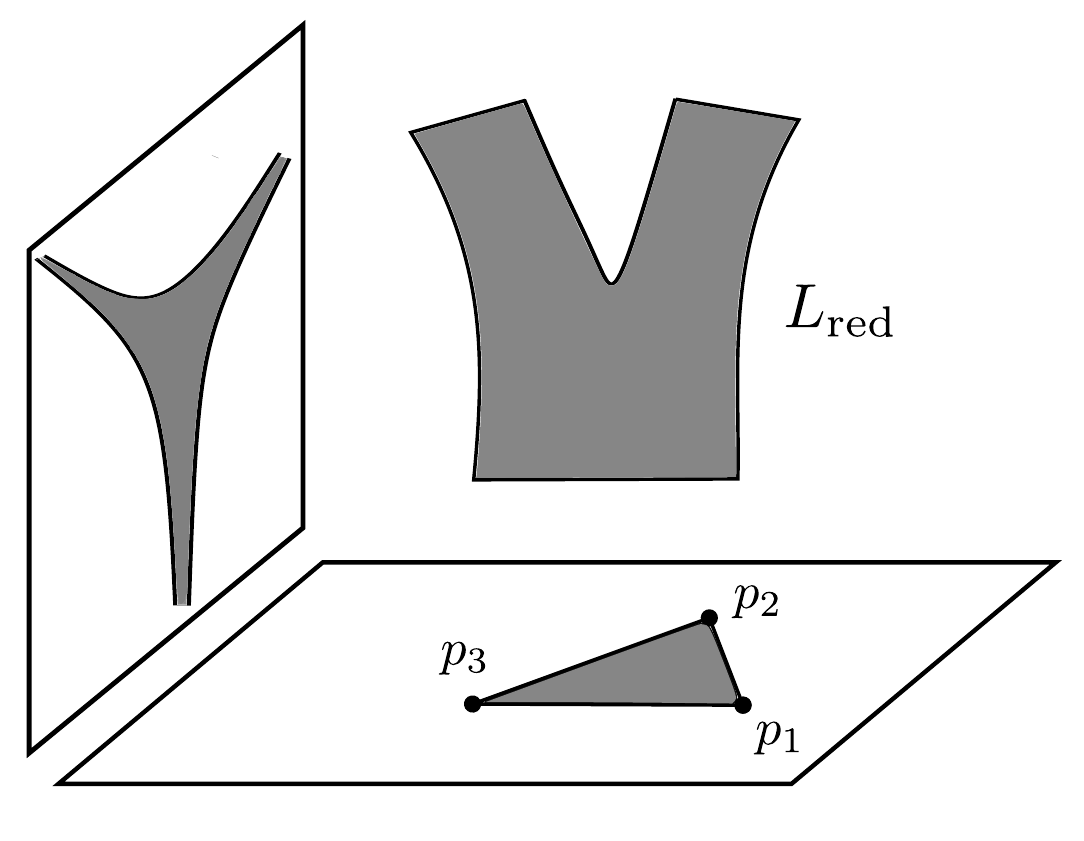}
\centering
  \caption{$L_{\text{red}}$ in $Z_{\text{red}}= \mathbb{R}^4$, in the case $n=3$.}\label{DSC}
\end{figure}
The equation $du_1 \wedge dy_1 + du_2 \wedge dy_2 = 0$ is equivalent to $\frac{\partial y_2}{\partial u_1} =  \frac {\partial y_1}{\partial u_2}$. This implies that we can define a function $\varphi: U \to \mathbb{R}$ such that $F = \nabla \varphi$, namely
\begin{align*}
    y_1 = \frac{\partial \varphi}{\partial u_1}, \quad \quad 
    y_2 =  \frac{\partial \varphi}{\partial u_2}.
\end{align*}
The equation $V du_1 \wedge du_2 - dy_1 \wedge dy_2 = 0$ reduces to the following real Monge-Amp\`ere equation for $\varphi$. 
\begin{align}\label{MA0}
    \det D^2 \varphi = V=A+\sum_{i=1}^n  \frac{1}{2|u-p_i|}.
\end{align}
Note that $V$ has singularities at the vertices of the convex polygon $U$. 

\textbf{Dirichlet boundary condition.}
We now use the expected asymptote of the special Lagrangian $L$ to heuristically motivate the Dirichlet boundary condition on the real Monge-Amp\`ere equation. The rest of the paper will start from the Dirichlet problem and construct the conjectured special Lagrangian $L$. In Section \ref{Asymptotic-behavior}, we will see that the Dirichlet boundary condition results in the correct asymptotic behavior. 

Notice that the vector $\tilde{v}_i\in \mathbb{R}^2$ is normal to the edge $[p_i, p_{i+1}]$. Since $L_{\text{red}}$ is the graph of $\nabla \varphi$ and its asymptotic at infinity is the union of $L_i$, we expect that when $u \in U$ approaches the open line segment $(p_i, p_{i+1})$ on the boundary of the convex polygon, the normal derivative tends to infinity, while the tangential derivative tends to a constant, $\nabla\varphi \cdot (p_{i+1}-p_i)\to c_i$. Since $\sum c_i=0$, we can write $c_i= b_{i+1}-b_{i}$, and notice that adding a global constant to $b_i$ will be inconsequential. 
Thus the boundary value of $\varphi$ on the edge is expected to be \emph{affine linear}, namely
\begin{equation}\label{boundarydata}
\varphi(p_i)=b_i, \quad 
\varphi(tp_i+ (1-t)p_{i+1})= t b_i+ (1-t)b_{i+1}, \quad \forall t\in [0,1].
\end{equation}

\begin{remark}\label{standardposition}
The real Monge-Amp\`ere equation is invariant under adding an affine linear function. Geometrically, adding a constant to $\varphi$ has no effect on the special Lagrangian $L$, and adding a linear function amounts to translating $L$ along some vector in $\mathbb{R}_{(y_1,y_2)}^2$. In particular, for the triangle case $n=3$, we can reduce to the zero boundary data.  

The real Monge-Amp\`ere equation is also invariant under the Euclidean motion of the convex polygon $U$, with the corresponding change to $V$. Later, when we analyze the local behavior of $\varphi$ near an open edge, we will sometimes reduce to the `standard position' where the open edge is contained in the $u_2$-axis, and $U$ lies inside the right half-plane, to simplify the notations.
\end{remark}

\subsection{Solving the Dirichlet problem}\label{RMAequation}

\begin{theorem}
[Dirichlet problem]
Let $b_1, \ldots, b_n \in \mathbb{R}$. There is a unique continuous convex function $\varphi: \overline{U} \to \mathbb{R}$, with boundary data (\ref{boundarydata}), which is smooth in $U$ and solves the real Monge-Amp\`ere equation (\ref{MA0}).
\end{theorem}

The remainder of this section is dedicated to the proof of this theorem. We use an approximation strategy to deal with the failure of strict convexity of the domain.

\begin{proof}
    \textbf{Step 1 (Approximate solutions).} 
   We take a smooth 1-parameter expanding family of strictly convex smooth domains $U_t\subset U$, where $t \in (0,1)$, converging to $U_1 = U$ as $t\to 1$, as in Figure \ref{domains}.
\begin{figure}[H]
\includegraphics[width=9cm]{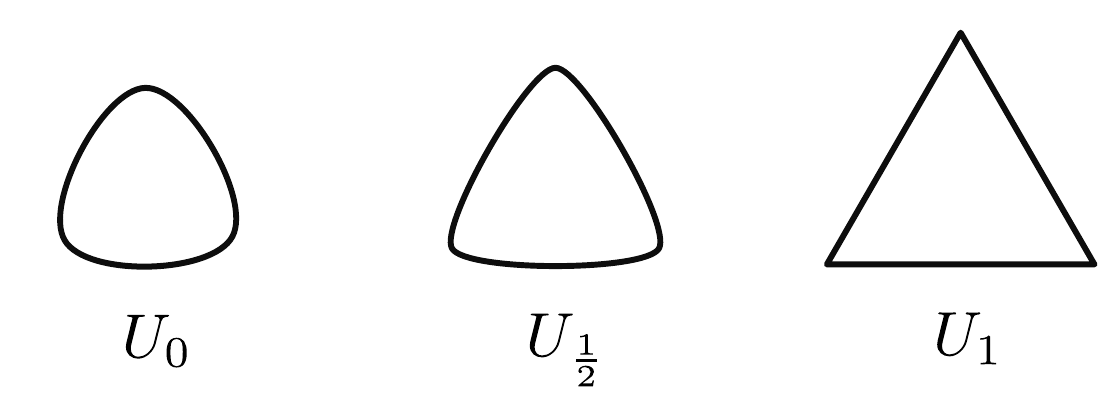}
\centering
  \caption{Approximating domains $U_t \to U$, in the case $n=3$.}
  \label{domains}
\end{figure}
The piecewise linear boundary data (\ref{boundarydata}) can be extended to some Lipschitz function $\bar{\phi}$ on $\overline{U}$. We consider the following Dirichlet problem for $\varphi_t : U_t \to \mathbb{R}$ for each $t \in (0,1)$.
\begin{equation}\label{Dir2}
\begin{cases}
\begin{alignedat}{3}
\det D^2 \varphi_t &= V, \quad &&\text{ on } && \quad U_t,\\
\varphi_t &= \bar{\phi}, &&\text{ on } && \quad \partial U_t.
\end{alignedat}
\end{cases}
\end{equation}

\begin{lemma}[Rauch-Taylor \cite{MR0454331}]\label{Rauch-Taylor-M-A}
Let $\Omega \subset \mathbb{R}^2$ be a strictly convex domain, $g: \partial \Omega \to \mathbb{R}$ a continuous function, and $\mu$ a non-negative Borel measure on $\Omega$ with $\mu(\Omega) < \infty$. Then, there exists a unique convex function $f \in C(\overline{\Omega})$ such that
\begin{equation}
\begin{cases}
\begin{alignedat}{3}
\det D^2 f &= \mu \quad &&\text{ on } && \quad \Omega,\\
f &= g, \quad &&\text{ on } && \quad \partial \Omega.
\end{alignedat}
\end{cases}
\end{equation}
\end{lemma}
In our case, for every $t \in (0,1)$, we let $\Omega = U_t$ and define $\mu=Vdu_1du_2$. Notice here that $V$ is strictly positive within $U_t$ and uniformly bounded in $L^1$,
\begin{align*}
    \int_{U_t} V du_1 du_2 \leq  \int_{U} V du_1 du_2 <+\infty, \quad \text{ for all } \quad t \in (0,1).
\end{align*}
Therefore, by Lemma \ref{Rauch-Taylor-M-A}, there is a unique convex function $\varphi_t \in C(\overline{U}_t)$ satisfying (\ref{Dir2}).

\textbf{Step 2 (Uniform bounds and the limit).} By the convexity of $\varphi_t$, and the Lipschitz property of $\bar{\phi}$, we obtain the uniform upper bound
\[
\varphi_t(u) \leq \bar{\phi}(u)+ C \text{dist}(u, \partial U_t),\quad \forall u\in U_t.
\]

We now derive a uniform lower bound.

\begin{lemma}\label{Lemma}
There is a uniform constant independent of $t$, such that 
\[
\varphi_t (u) \geq  \bar{\phi}(u)- C \text{dist}(u, \partial U)^{1/2},\quad \forall u\in U_t.
\]
\end{lemma}

\begin{proof} 
As in Remark \ref{standardposition}, we put $U$ into the standard position, namely that the boundary edge of interest $[p_i,p_{i+1}]$ lies on the $u_2$-axis, and $U$ lies in the right half-plane, as in Figure \ref{rectangle}, and without loss of generality, $\bar{\phi}=0$ on this boundary edge. It suffices to show $\varphi_t(u)\geq -Cu_1^{1/2}$ for some constant $C$ independent of $t$.

By the Lipschitz property of $\bar{\phi}$, we know $|\bar{\phi}(u)|\leq C_1 u_1$ on $U$, for some constant $C_1$. The boundary data of $\varphi_t+C_1u_1$ is non-negative, and it solves the same real Monge-Amp\`ere equation (\ref{MA0}).

By the Alexandrov estimate,
\[
\max(0, - \varphi_t(u))^2 \leq C_2 \text{diam}(U_t) \cdot \text{dist} (u, \partial U_t)  \cdot \int_{U_t} V du_1 du_2 \leq C_3 u_1, \quad \text{ for all } \quad u \in \overline{U}_t,
\]
as required.
\end{proof}

Combining the upper and lower bounds, we obtain a uniform bound on the $C^0$-norm of $\varphi_t$. By the convexity of $\varphi_t$, this implies a uniform Lipschitz bound on any fixed compact subset of $U$, as $t\to 1$. By Arzel\`a-Ascoli, we can extract a subsequence of $\varphi_t$ which $C^0$-converges to some continuous function $\varphi$ on any compact subset of $U$, which is a viscosity solution to the real Monge-Amp\`ere equation. By passing the upper and lower bounds to the limit as $t\to 1$, we obtain
\begin{equation}\label{boundarycontinuity}
\bar{\phi}(u)- C\text{dist}(u, \partial U)^{1/2} \leq \varphi(u)\leq \bar{\phi}(u)+ C\text{dist}(u, \partial U). 
\end{equation}
Thus $\varphi$ extends continuously to $\overline{U}$ and achieves the boundary data (\ref{boundarydata}), which agrees with $\bar{\phi}$ on $\partial U$. The uniqueness of the solution is a standard consequence of the maximum principle.

\textbf{Step 3 (Interior smoothness).} 
Notice $V$ is smooth and strictly positive in $U$.
As a standard fact about real Monge-Amp\`ere equation in dimension two, the solution $\varphi$ to the real Monge-Amp\`ere equation is smooth in the interior domain $U$. This fact is the consequence of two standard results: Caffarelli \cite{MR1038359} proved that the singular set must propagate to the boundary along some line segment, while Mooney's partial regularity \cite{MR3340380} showed that the singular set has codimension one Hausdorff measure zero. 
\end{proof}

\subsection{Gradient divergence near the edges}

In this section, we study the behavior of $\varphi$ near an open edge of the convex polygon $U$. The following theorem is essential in proving the smoothness of the special Lagrangian $L$.

\begin{proposition}\label{prop:gradientdivergence}
\label{blowup}
Let $u_*$ be any point on a boundary edge $(p_i, p_{i+1})$ in $\partial U\setminus \{ p_1,\ldots, p_n\}$. Then, as $u\in U$ tends to $u_*$, the normal and tangential gradient components satisfy
\begin{alignat*}{3}
        \nabla^{\text{normal}} \varphi (u) &\to +\infty, 
        \quad
         \nabla \varphi (u)\cdot (p_{i+1}-p_i) &\to c_i.
\end{alignat*}
\end{proposition}

\begin{proof}\label{blowupproof}
The tangential component converges to a constant by the convexity of $\varphi$, the affine linearity of the boundary data, and the boundary continuity estimate (\ref{boundarycontinuity}) by considering the convex function restricted to line segments parallel to the boundary edge. 

We now focus on the normal gradient component and place the convex polygon into the standard position by Remark \ref{standardposition}, so the open boundary edge $(p_i, p_{i+1})$ containing $u_*$ lies on the $u_2$-axis, the domain is contained in the right half-plane, and the boundary data is zero on this edge. The normal gradient component is just $-\partial_{u_1}\varphi$. Using convexity, $\partial_{u_1}\varphi$ is bounded from the above near $u_*$.

We suppose for contradiction that $\partial_{u_1} \varphi (u)$ stays bounded for some sequence $u \to u_*$, and therefore, the gradient at $u$ stays bounded. Then there exists some subgradient for $\varphi$ at $u_*$, which must be of the form $(-\Lambda, 0)$ for some $\Lambda\in \mathbb{R}$, because the tangential component is zero. 
Let 
\begin{align*}
    \Lambda_0 = \inf \{ \Lambda \; | \; (-\Lambda,0) \text{ is a subgradient of } \varphi \text{ at some point } v \in (p_i,p_{i+1})  \},
\end{align*}
so in particular $\varphi(u)\geq -\Lambda_0 u_1$ for any $u\in \overline{U}$. Thus $(-\Lambda_0, 0)$ is a subgradient at every point on the open edge, and $\partial_{u_1} \varphi\geq -\Lambda_0$ on $U$.

We fix a small constant $h > 0$ such that $u_*$ has distance at least $2h$ to the vertices. Let $R(u_*, h, \varepsilon)$ be the rectangle with length $h$ and width $\varepsilon\ll 1$, as shown in Figure \ref{rectangle}.
    \begin{figure}[H]
    \includegraphics[width=5cm]{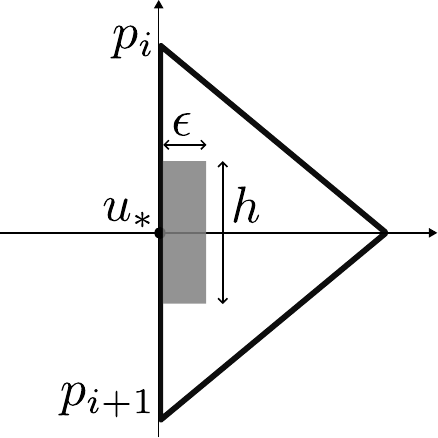}
    \centering
    \caption{Standard position in the case $n=3$.\\ $R(u_*,h, \varepsilon)$ is the shaded region.}
    \label{rectangle}
    \end{figure}
By the Monge-Amp\`ere equation, and the strict positivity of $V$,
\begin{align}\label{equationV}
  C^{-1}h\varepsilon  \leq \int_{R(u_*,h, \varepsilon)} V du_1 du_2 = \int_{\nabla \varphi(R(u_*,h, \varepsilon))} dy_1 dy_2.
\end{align}
By the convexity of $\varphi$ and the bound $-\Lambda_0 u_1\leq \varphi\leq Cu_1$, we deduce
\begin{align*}
    |\partial_{u_2} \varphi(u) | \leq C (\Lambda_0,h) u_1, \quad \text{ for all } \quad u \in R(u_*,h, \varepsilon).
\end{align*}
Thus by considering the gradient image,
\begin{align*}
    |\int_{\nabla \varphi(R(u_*,h, \varepsilon))} dy_1 dy_2| \leq C( \Lambda_0, h )\varepsilon  \sup_{ R(u_*,h, \varepsilon) }(\partial_{u_1} \varphi(u) + \Lambda_0),
\end{align*}
Contrasting with (\ref{equationV}), for any small $\varepsilon>0$,
\begin{align*}
   \sup_{ R(u_*,h, \varepsilon) }(\partial_{u_1} \varphi(u) + \Lambda_0)  \geq C(\Lambda_0, h)^{-1}.
\end{align*}
In the limit $\varepsilon\to 0$, we can extract some subgradient at some boundary point, which contradicts the minimality of $\Lambda_0$.
\end{proof}

\subsection{Solutions near vertices}\label{soolutions-near-vertices}

In this section, we examine the behavior of the gradient of $\varphi$ near the vertices of the convex polygon, which will be important in studying the smoothness of $L$. The ideal picture to have in mind, which we justify in this section, is shown in Figure \ref{mapping} for the case $n=3$. The Map $F = \nabla \varphi$ takes the bounded gray solid convex polygon to the unbounded gray area.
\begin{figure}[H]
\includegraphics[width=14cm]{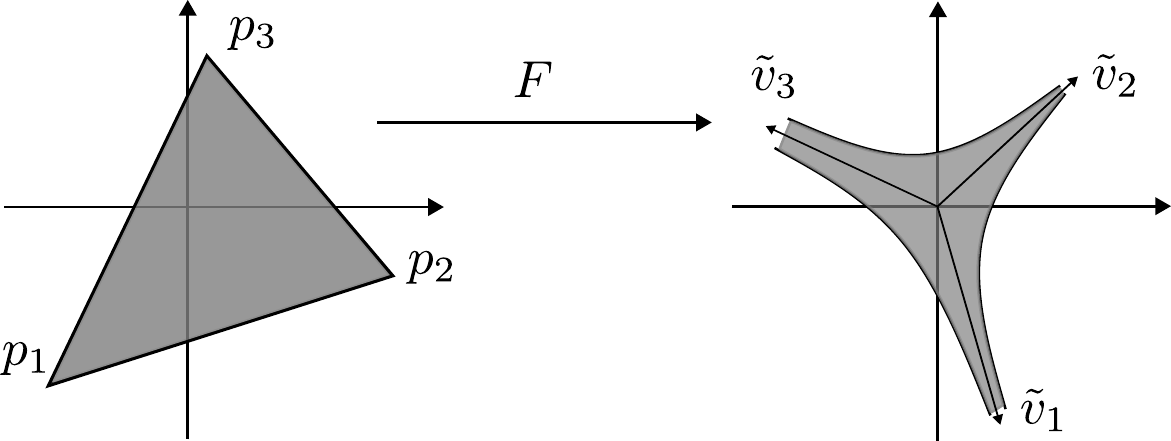}
    \centering
    \caption{Mapping $F = \nabla \varphi: U \subset \mathbb{R}^2_{(u_1,u_2)} \to \mathbb{R}^2_{(y_1,y_2)}$.}
    \label{mapping}
\end{figure}
Denote the subgradient sets at the vertices by
\begin{align*}
    C_{p_i} = \{ y \in \mathbb{R}^2_{(y_1,y_2)} \; | \; \varphi(u)-\varphi(p_i) \geq \langle y,  u - p_i\rangle \text{ for all } u \in \overline{U}\}.
\end{align*}

\begin{lemma}\label{disjoint}
The sets $C_{p_i}$ are disjoint convex closed subsets of $\mathbb{R}^2$ contained in the wedge region
\begin{align*}
    C_{p_i} \subset W_{p_i}=\{ y \in \mathbb{R}^2 \; | \; y \cdot (p_{i+1}-p_{i}) \leq b_{i+1}-b_i \} \cap \{ y \in \mathbb{R}^2 \; | \; y \cdot (p_{i-1}-p_i) \leq b_{i-1}-b_i \}.
\end{align*}
\end{lemma}

\begin{remark}\label{translatedwedge}
    The wedge region is a translated copy of the wedge
\begin{equation*}
W_{p_i}'= \{ y \in \mathbb{R}^2 \; | \; y \cdot (p_{i+1}-p_{i}) \leq 0 \} \cap \{ y \in \mathbb{R}^2 \; | \; y \cdot (p_{i-1}-p_i) \leq 0 \},
\end{equation*}
in which the directions of its two extremal rays are specified by the vectors $\tilde{v}_{i-1}, \tilde{v}_i$. For different $i$, the wedges $W_{p_i}'$ can only intersect along boundary rays, and the intersections between different $W_{p_i}$ have areas bounded by some constant depending only on $\{ b_k\}_{k=1}^n$.
\end{remark}

\begin{figure}[H]
    \includegraphics[width=15cm]{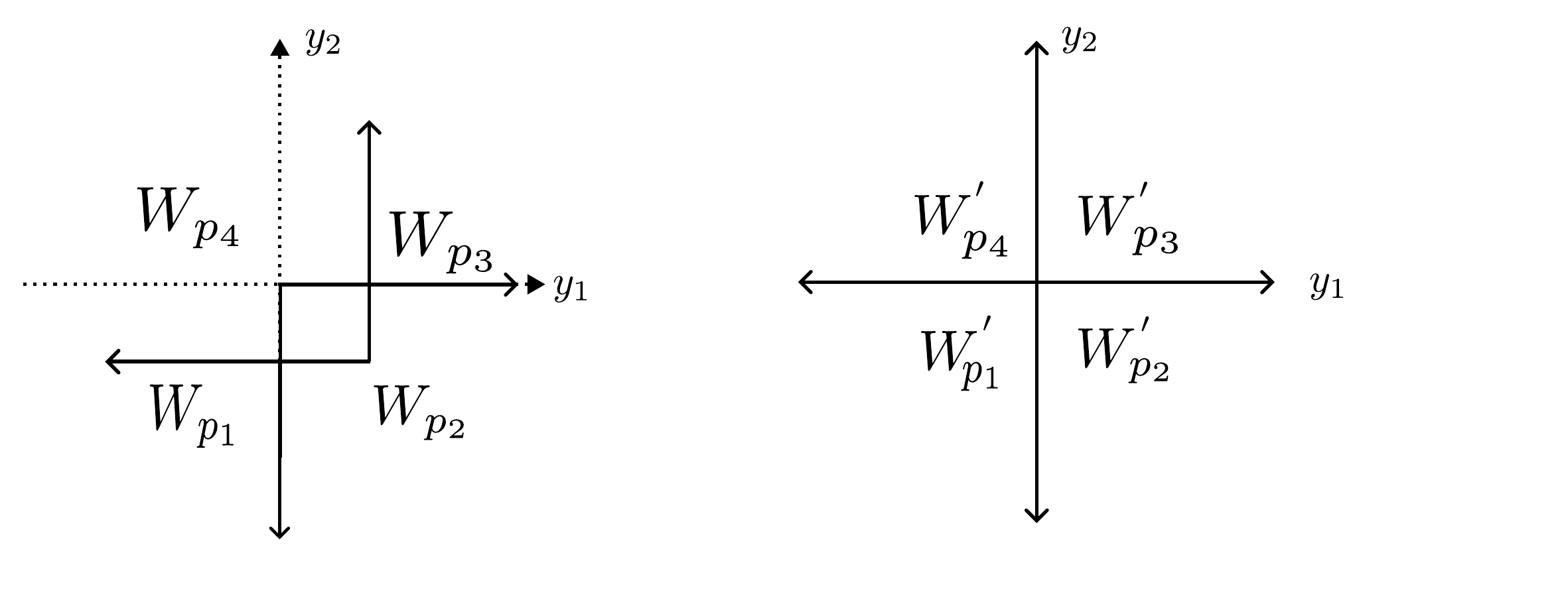}
    \centering
    \caption{Adjacent wedges $W_{p_i}$ and $W_{p_{i+1}}$ intersect along a common ray, but non-adjacent wedges may intersect nontrivially inside a compact set.}
    \label{example}
\end{figure}

\begin{proof}[Proof of Lemma \ref{disjoint}]
Suppose $C_{p_i}$ and $C_{p_j}$ have a common vector, then by convexity, this vector is a subgradient for all the points on the line segment $[p_i, p_j]$. If the open segment $(p_i, p_j)$ is in the interior domain $U$, this would contradict the strict convexity of $\varphi$ (which follows from the smoothness of the real Monge-Amp\`ere solution), and if $(p_i, p_j)$ is an edge in the boundary, this would imply the existence of subgradient at every point of $(p_i,p_j)$, contradicting Prop. \ref{prop:gradientdivergence}.
 This proves the disjointness of $C_{p_i}$. 
\end{proof}

\begin{figure}[H]
    \includegraphics[width=6cm]{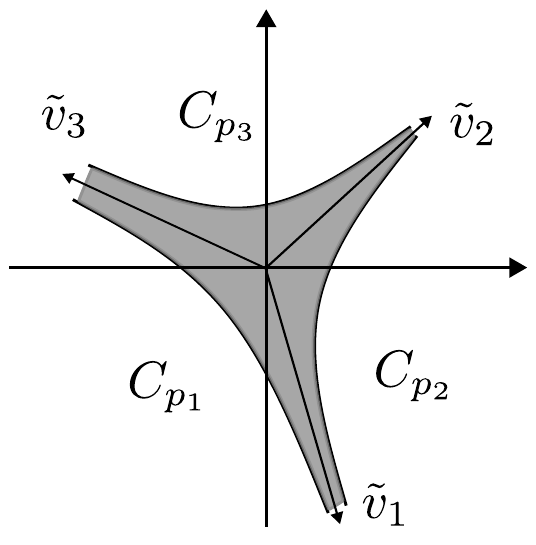}
    \centering
    \caption{ Subgradients $C_{p_1}, C_{p_2}$, and $C_{p_3}$, in the case $n=3$.}
    \label{subgradients}
\end{figure}

\begin{lemma}[Image of the gradient]
    We have $\nabla \varphi(U) = \mathbb{R}^2 \setminus (\cup_i C_{p_i} )$.
\end{lemma}

\begin{proof}
First we claim that $\mathbb{R}^2= \cup_i  C_{p_i} \cup \nabla \varphi(U)$. Given any $y\in \mathbb{R}^2$, we consider the graph of the affine linear function $a+ \langle y, u\rangle$ on $\overline{U}$,  where $a\in \mathbb{R}$ increases from negative infinity. There must be some $a$ when the graph first touches the graph of the convex function $\varphi$. This shows that $y$ is a subgradient at some point on $\overline{U}$. But the divergence of the gradient on the edge (Prop. \ref{blowup}) shows that there is no subgradient at any point on the open edges, so $y$ is a subgradient either at an interior point or at one of the vertices.

By the strict convexity of $\varphi$ in $U$, we see \begin{align*}
\nabla \varphi(U) \cap C_{p_i} = \emptyset.
\end{align*}
Thus  $\nabla \varphi(U) = \mathbb{R}^2 \setminus (\cup_i C_{p_i} )$.
\end{proof}

\begin{lemma}\label{close-to-vertices}
    For any sequence of points $u\in U$ converging to $ p_i$, after passing to a subsequence, either $\nabla \varphi (u) \to \infty$ or $\nabla \varphi (u)$ converges to a point in $\partial C_{p_i}$. 
\end{lemma}

\begin{proof}
Suppose $|\nabla \varphi (u)|$ stays bounded. After passing to a subsequence, $|\nabla \varphi (u)|$ converges to some $y$. Then $y$ is a subgradient at $p_i$, so lies in $C_{p_i}$. However, the limit $y$ has to lie in the closure of $\nabla\varphi(U)$, which is disjoint from the interior of $C_{p_i}$. Thus the gradient lies on $\partial C_{p_i}$.
\end{proof}

Lemma \ref{infinite-wedge} gives some asymptotic decay bound on the gradient image $\nabla \varphi(U)$.

\begin{lemma}\label{infinite-wedge}
Suppose $y\in \nabla\varphi(U)$ and $|y|\geq 1$. Then, for one of the boundary rays $\mathcal{R}$ of some wedge region $W_{p_i}$, we have $\text{dist}(y, \mathcal{R}) \leq C/|y|$, where the distance is measured in the Euclidean $\mathbb{R}^2$, and the constant $C$ is independent of $y$.
\end{lemma}

In particular, Lemma \ref{infinite-wedge} shows that a scenario similar to the one shown in Figure \ref{wedge}, where $\nabla \varphi (U)$ contains an infinite wedge, cannot happen.

\begin{figure}[H]
    \includegraphics[width=6cm]{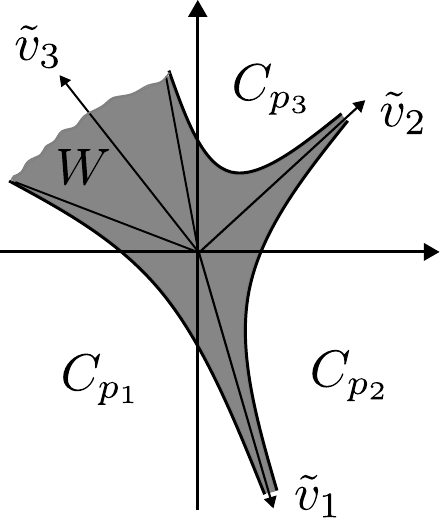}
    \centering
    \caption{$\nabla \varphi (U)$ cannot contain an inifinite wedge $W$.}
    \label{wedge}
\end{figure}

\begin{proof}
Without loss of generality, we assume that $y$ is large compared to $\max_i |b_i - b_{i+1}|$. The rays $\mathbb{R}_{\geq 0} \tilde{v}_i$ partition $\mathbb{R}^2$ into wedge-shaped regions $W_{p_i}'$. We choose the direction $\tilde{v}_i$ which minimizes the angle with the direction of $y$. Notice $\tilde{v}_i$ is parallel to the boundary ray between  $W_{p_i}$ and $W_{p_{i+1}}$, and by the choice of $\tilde{v}_i$ and the largeness of $y$, we see $y$ must lie in either $W_{p_i}$ or $W_{p_{i+1}}$, and without loss we focus on $y\in W_{p_i}$.

Let $q_i$ be the vertex point of $W_{p_i}$, and let $\mathcal{R} = \mathbb{R}_{\geq 0} \tilde{v}_i + q_i$, and let $y'$ be the intersection point of the ray $\mathcal{R}$ with the ray $y+ \mathbb{R}_{\leq 0} \tilde{v}_{i-1}$. If $y=y'$, then the distance to the ray is zero, and we are done. So without loss $y\neq y'$, and we
consider triangle $T(y)$ with vertices at $y, y', q_i$. 

\begin{figure}[H]
    \includegraphics[width=5cm]{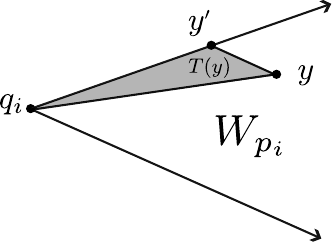}
    \centering
    \caption{Triangle $T(y)$.}
    \label{width}
\end{figure}

By construction the triangle $T(y)$ is contained in $W_{p_i}$. By elementary geometry, the area of this triangle is comparable to $|y| \text{dist}(y, \mathcal{R})$ when $|y|$ is much larger than $\max_i |b_i-b_{i+1}|$.

Using the partition $\mathbb{R}^2= \nabla \phi(U)\cup \cup_j C_{p_j}$, 
 \begin{equation}
 \begin{split}
   \text{Area}(W_{p_i}\setminus C_{p_i})  & \leq \int_{\nabla \varphi(U)} dy_1 dy_2 +  \text{Area}(W_{p_i}\cap \cup_{j\neq i} C_{p_j})
   \\
   &\leq \int_{\nabla \varphi(U)} dy_1 dy_2 +
   \sum_{j\neq i}  \text{Area}(W_{p_i}\cap W_{p_j} )   \\
   & \leq \int_U Vdu_1du_2+ C \leq C.
 \end{split}
 \end{equation}
Here the second line uses $C_{p_j}\subset W_{p_j}$, and the third line uses the Monge-Amp\`ere equation and the fact that non-adjacent wedge regions $W_{p_j}$ only intersect within a compact set depending only on  $\{b_k\}$ and the direction of the rays. Since $T(y)\subset W_{p_i}$, in particular $ \text{Area}(T(y)\setminus C_{p_i}) \leq C$.

We claim that $C_{p_i}\cap T(y)$ is empty.
Suppose $y''\in C_{p_i}\cap T(y) \subset W_{p_i}$. Then by the convexity of $C_{p_i}$, and the fact that  $y\notin C_{p_i}$ because of $y\in \nabla \varphi(U)\subset \mathbb{R}^2\setminus \cup C_{p_i}$, we deduce that the ray starting from $y$ in the direction $y-y''$, lies in the complement of $C_{p_i}$. Since the direction of $y-y''$ lies in the wedge region $W_{p_i}'$,  by the convexity of $C_{p_i}$, the set $W_{p_i}\setminus C_{p_i}$ contains some infinite wedge region with vertex at $y$. This is impossible by the finiteness of $\text{Area}(W_{p_i}\setminus C_{p_i})$.

Combining the above,
\[
|y| \text{dist}(y, \mathcal{R}) \leq C \text{Area}(T(y)) \leq C, 
\]
hence $\text{dist}(y, \mathcal{R}) \leq \frac{C}{ |y|  }$ when $|y|$ is large.

    \vspace{5pt}

\end{proof}

\section{Regularity, asymptotics, and topology}

This section aims
to prove Theorem \ref{generalizationDonaldsonScaduto}, by producing the desired special Lagrangian from the solution of the real Monge-Amp\`ere equation, and establishing its smoothness, asymptotic properties, and topology.

\subsection{Smoothness of the special Lagrangian}\label{tangentcone}

Let $q: u_3^{-1}(0) \to Z_{\text{red}}$ be the quotient map, and $L^{\circ} = q^{-1}(\text{Graph}(\nabla \varphi
_{|_{U}}))$. By the smoothness of the real Monge-Amp\`ere solution $\varphi$, clearly $L^{\circ}$ is a smooth special Lagrangian submanifold in $Z$, diffeomorphic to  $\mathbb{D}^2 \times S^1$; however, it is not a closed subset of $Z$. Let $L$ be the current of integration defined by $L^\circ$, so that $L$ contains also points in the closure of $L^\circ$. The goal of this section is to prove the following theorem.

\begin{theorem}\label{smooth-submanifold}
    $L$ is represented by a smooth submanifold of $Z$ without boundary.
\end{theorem}

The proof of Theorem \ref{smooth-submanifold} follows from Lemma \ref{no-boundary} and Lemma \ref{smoothness}.

\begin{lemma}\label{no-boundary}
    $L$ is a closed integral current,  $\partial L = 0$. Morever, any point in the support of $L$ either lies on $L^\circ$ or lies above some vertex $p_i$.
\end{lemma}

\begin{proof}
We first show that $L$ has locally finite mass inside any given ball in $Z$. Since the gradient $\nabla \varphi$ diverges to infinity near the open edges $(p_i,p_{i+1})$, we only need to show that mass cannot accumulate near the vertices $p_i$. Now by the special Lagrangian condition, for any $U(1)$-invariant compact set $K\subset Z$,
\[
\text{Mass}(L\cap K) =\int_{L^\circ\cap K} \text{Re}(\Omega) =2\pi\int_{(L^\circ \cap K)/U(1)} (du_1\wedge dy_2-du_2\wedge dy_1).
\]
Since $\varphi$ is a smooth strictly convex function, both integrands are positive, and the maps $(u_1, u_2)\mapsto (u_1, y_2)$ and $(u_1, u_2)\mapsto (y_1, u_2)$ are injective, so $\int_{(L^\circ \cap K)/U(1)} du_1\wedge dy_2$ is simply the Lebesgue measure of the $(u_1,y_2)$ projection of ($L^\circ \cap K)/U(1)$, and likewise with the second integrand. 
By the boundedness of the range of $(u_1, y_2), (u_2,y_1)$ on $L^\circ \cap K$, we deduce that the integral is finite.

The support of $L$ is contained in the closure of $L^\circ$.
Since the gradient $\nabla \varphi$ diverges to infinity, there cannot be any limiting point in the support of $L$ whose projection to $\mathbb{R}^2_{(u_1,u_2)}$ lies above the open edges $(p_i,p_{i+1})$. Thus, the only points added in the closure lie above the vertices $p_i$, and by Lemma \ref{close-to-vertices}, they lie above $\{p_i\}\times \partial C_{p_i}$.

Let $\psi$ be any smooth and compactly supported test 2-form on $Z$, and for small $r>0$, let $\chi_r$ be a cutoff function supported in the union of  $B_{X^4}(p_i, r)\times \mathbb{R}^2\subset Z$, with value one on the union of $B_{X^4}(p_i,r/2)\times \mathbb{R}^2$, and the $g_X$-gradient of $\chi_r$ is bounded by $Cr^{-1}$. Since $(1-\chi_r)\psi$ is supported away from the $U(1)$-fixed locus, integration by parts gives
\[
\int_{L} \chi_rd \psi= - \int_L d\chi_r \wedge \psi= - \int_{L^\circ} d\chi_r \wedge \psi ,
\]
hence
\[
|\int_{L} \chi_rd \psi|\leq C(\psi) r^{-1} \sum_{i=1}^n \text{Mass}(L^\circ\cap  \text{supp}   (\psi)\cap B_{X^4}(p_i,r)\times \mathbb{R}^2) 
\]
Now $y_1,y_2$ are bounded within the support of $\psi$, and from the Gibbons-Hawking ansatz we know $|u-p_i|\leq C r^2$ on $B_{X^4}(p_i,r)$.
The same argument as in the local finiteness of measure now gives a bound $|\int_{L} \chi_rd \psi|\leq C(\psi)r$. Taking the limit $r\to 0$, we deduce that $\int_{L} d \psi=0$ for any test 2-form, which means $\partial L=0$.
\end{proof}

Now, we prove the smoothness. $L$ is a special Lagrangian integral current and, in particular, a minimal integral current. Therefore, there exists a tangent cone at each point $x$ on $L$. The proof of smoothness is based on the following implication of Allard's regularity theorem.

\begin{proposition}\label{smooth-points}
    A point $x \in \text{supp}(L)$ is a smooth point if and only if every tangent cone $N \subset \mathbb{C}^3$ at $x$ is a $3$-plane with multiplicity one.
\end{proposition}

Let $q(x) \in \{p_i\} \times \partial C_{p_i}$. Any tangent cone $N \subset \mathbb{C}^3$ at $x$ is a $U(1)$-invariant tangent cone in $\mathbb{C}^3$. To prove every tangent cone of $L$ is a $3$-plane with multiplicity one, we employ Joyce's classification of $U(1)$-invariant special Lagrangian cones in $\mathbb{C}^3$ \cite{MR2122282}.

\begin{proposition} [Joyce \cite{MR2122282}, Haskins \cite{MR2075484}]\label{Joyce-Haskins} 
Let $N$ be a special Lagrangian cone without boundary in $\mathbb{C}^3$ invariant under the U(1)-action given by
    \begin{align*}
        e^{i \theta} : (z_1, z_2, z_3) \to (e^{i \theta} z_1, e^{-i \theta} z_2, z_3), \quad \text{for} \quad e^{i \theta} \in U(1),
    \end{align*} 
where $N \setminus \{0\}$ is connected. Then there exists $A \in [-1, 1]$ and functions $w : \mathbb{R} \rightarrow (-1,1)$, and $\alpha, \beta : \mathbb{R} \rightarrow \mathbb{R}$ satisfying the following system of differential equations:
\begin{alignat*}{2}
&(\frac{dw}{dt})^2 = 4((1-w)^2(1+2w)-A^2),  \quad \quad &&\frac{d\alpha}{dt} = \frac{A}{1-w}, \\
&\frac{d\beta}{dt} = \frac{-2A}{1+2w},  &&(1-w)(1+2w)^{\frac{1}{2}}\cos(2\alpha+\beta) = A,
\end{alignat*}
such that, away from points $(z_1, z_2, z_3) \in \mathbb{C}^3$ with $z_j = 0$ for some $j$, we may locally write $N$ in the form $\Phi(r, s, t) : r > 0, s, t \in \mathbb{R}$, where
\[
\Phi : (r, s, t) \mapsto (re^{i(\alpha(t)+s)} \sqrt{1 - w(t)}, e^{i(\alpha(t)-s)}\sqrt{1 - w(t)},  re^{i\beta(t)} \sqrt{1 + 2w(t)}),
\]
and exactly one of the following holds:
\begin{itemize}
  \item $A = 1$. Then, $N$ is the $U(1)^2$-invariant special Lagrangian $\mathbb{T}^2$-cone 
  \begin{align*}
      \{ (r e^{i\theta_1}, r e^{i\theta_2}, r e^{i\theta_3}) \quad | \quad r > 0, \quad \theta_1, \theta_2, \theta_3 \in \mathbb{R}, \quad \theta_1 + \theta_2 + \theta_3 = 0
      \}.
  \end{align*}

    \item $A = -1$. Then, $N$ is the $U(1)^2$-invariant special Lagrangian $\mathbb{T}^2$-cone 
    \begin{align*}
      \{ (r e^{i\theta_1}, r e^{i\theta_2}, r e^{i\theta_3}) \quad | \quad r > 0, \quad \theta_1, \theta_2, \theta_3 \in \mathbb{R}, \quad \theta_1 + \theta_2 + \theta_3 = \pi
      \}.
    \end{align*}

  \item $A = 0$. Then, for some $\phi \in (-\pi, \pi]$, either $N = \Pi_{\phi}^+$ or $N = \Pi_{\phi}^-$ or $N$ is the singular union $\Pi_{\phi}^+ \cup \Pi_{\phi}^-$, where $\Pi_{\phi}^{\pm}$ are the special Lagrangian 3-planes
  \begin{align*}
      \Pi_{\phi}^+ = \{(z,ie^{-i\phi}\overline{z},re^{i\phi}) | z \in \mathbb{C}, r \in \mathbb{R}\}, \; \text{ and } \; \Pi_{\phi}^- = \{(z,-ie^{-i\phi}\overline{z},re^{i\phi}) | z \in \mathbb{C}, r \in \mathbb{R}\}.
  \end{align*}
  
  \item $0 < |A| < 1$. Then, the function $w(t)$ may be written in terms of the Jacobi elliptic functions. It is non-constant and periodic in $t$ with period $T$ depending only on $A$, and $2\alpha + \beta$ is also non-constant and periodic in $t$ with period $T$. 
\end{itemize}
\end{proposition}

We proceed by ruling out every possibility on Joyce's list except 3-planes. We do this using the following lemma.

\begin{lemma}\label{tangent-cone-wedge}
Let $N \subset \mathbb{C}^3$ be a special Lagrangian tangent cone of $L$ at $x$, where $q(x) \in p_i \times \partial C_{p_i}$. Let $U = \pi_1(N/U(1)) \subset \mathbb{R}_{(u_1,u_2)}^2$. The set  $U$ is a subset of the infinite wedge with vertex $p_i$ and two rays along the direction $\overrightarrow{p_ip_{i+1}}$ and $\overrightarrow{p_ip_{i-1}}$.

\end{lemma}

\begin{proof}
The image of $L$ in $\mathbb{R}^2_{(u_1,u_2)}$ is a convex polygon. Let $W$ be the infinite wedge with vertex at $p_i$, and two boundary rays $\overrightarrow{p_ip_{i+1}}$ and $\overrightarrow{p_ip_{i-1}}$. In particular its openning angle is less than $\pi$.  
 The $\pi_1$-projections of all the special Lagrangians obtained by rescaling $L$ around the base point $x$, are all contained in $W$, so by passing to the limit,  the same holds for the projection of the tangent cone $N$ to $\mathbb{R}^2_{(u_1,u_2)}$.
\begin{figure}[H]
\includegraphics[width=10cm]{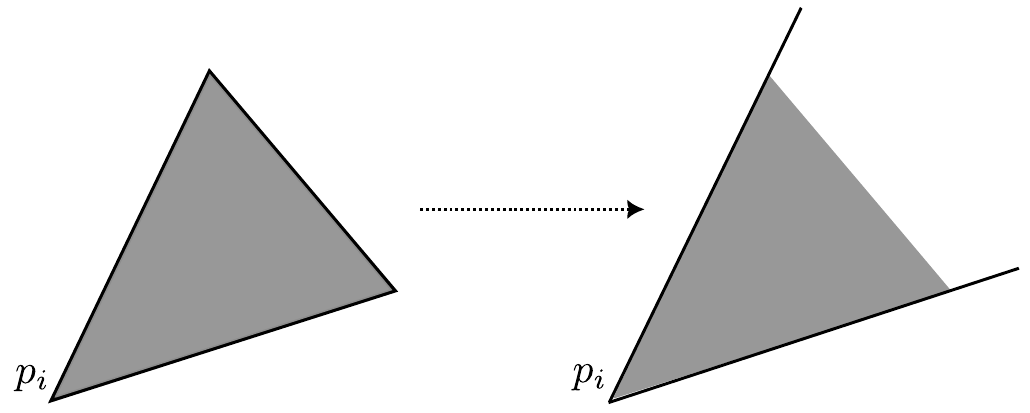}
\centering
  \caption{The $\pi_1$-projection of a tangent cone at $p_i$.}
  \label{tangent-cone}
\end{figure}
\end{proof}

\begin{lemma}\label{smoothness}
    Let $N \subset \mathbb{C}^3$ be a special Lagrangian tangent cone at $x$ in $L$, where $q(x) \in \{p_i\} \times \partial C_{p_i}$. Then, $N$ is a 3-plane with multiplicity 1.
\end{lemma}

\begin{proof} We apply Joyce's classification to the connected components of the tangent cone, and rule out every other possibility of the list of Proposition \ref{Joyce-Haskins}.

\textbf{Step 1 ($U(1)^2$-invariant $\mathbb{T}^2$-cone).} The cases $A = 1$ and $A = -1$ are similar, so we focus on $A = 1$. In this scenario,
\begin{align*}
    \pi_1(N/U(1)) = \{ r^2e^{-i\theta_3} \; | \; r > 0, \; \theta_3 \in \mathbb{R} \}.
\end{align*}
$r^2e^{-i\theta_3}$ can take any value in $\mathbb{C} = \mathbb{R}^2_{(u_1,u_2)}$, and consequently, $\pi_1(N/U(1)) = \mathbb{R}^2$. In particular, $\pi_1(N)$ is not subset of a wedge with angle less that $180^{\circ}$, which contradicts Lemma \ref{tangent-cone-wedge}.

\textbf{Step 2 (Union of two 3-planes).}  Suppose $\text{supp}(N)$ contains  $\Pi^{\phi}_+ \cup  \Pi^{\phi}_-$. We have $u_1(z, ie^{-i \varphi} \overline{z}) = \text{Re}(z_1z_2) = |z|^2 \sin (\varphi)$ and $u_2(z, ie^{-i \varphi} \overline{z}) = \text{Im}(z_1z_2) = |z|^2 \cos (\varphi)$. Hence
\begin{align*}
    \pi_1(\Pi^{\phi}_+) &= \{ r^2 (\sin (\varphi), \cos(\varphi)) \; | \; r \in \mathbb{R} \} \quad \text{ and } \quad \\\pi_1(\Pi^{\phi}_-) &= \{ -r^2 (\sin (\varphi), \cos(\varphi)) \; | \; r \in \mathbb{R} \}.
\end{align*} 
Therefore, $\pi_1(N/U(1)) = \{ R (\sin (\varphi), \cos(\varphi)) \; | \; R \in \mathbb{R}\}$ forms a line. In particular, it is not subset of a wedge with angle less than $180^{\circ}$, contradicting Lemma \ref{tangent-cone-wedge}.

\textbf{Step 3 (Multiplicity and graphicality).} The special Lagrangian $L$ projects with degree one to the $ (y_1,u_2)$-plane.  
As in Theorem 5.6. in \cite{MR2122282}, for any given cutoff function $\chi\geq 0$ on $\mathbb{R}^2$,
\begin{align*}
    \int_L \chi (y_1, u_2)  (\theta du_1 - V du_2 du_3) \wedge dy_1 = 2 \pi \int_{S} \chi(y_1,u_2) dy_1 du_2,
\end{align*}
where $S$ is the projection of $L/U(1)$ on the $(y_1, u_2)$-plane. By passing to a tangent cone $N$ of $L$ at $x$, we get:
\begin{align*}
    \int_N \chi (\text{Re}(z_3), \text{Im}(z_1 z_2)) \text{Re}(dz_1 \wedge dz_2) \wedge \text{Re}(dz_3) \leq 2 \pi \int_{\mathbb{R}^2} \chi(y_1,u_2) dy_1 du_2.
\end{align*}

In a small open neighborhood around a generic point $(y_1, u_2) \in \mathbb{R}^2$, the tangent cone $N$ is nonsingular and divides into $k$ components, therefore for $\chi$ supported near this point we get 
\begin{align*}
    \int_N \chi (\text{Re}(z_3), \text{Im}(z_1 z_2)) \text{Re}(dz_1 \wedge dz_2) \wedge \text{Re}(dz_3) = 2 \pi k \int_{\mathbb{R}^2} \chi(y_1,u_2) dy_1 du_2.
\end{align*}
Therefore, $k = 0$ or $1$. Thus the preimage of $(y_1, u_2)\in \mathbb{R}^2$ contains at most one point counting multiplicity. The same conclusion holds for the projections to $(y_2, u_1)$, and indeed any choice of direction of $y$ and the corresponding direction of $u$ specified by the partial Legendre transform. 
This forces that
there is at most one connected component in $N$, and it has multiplicity one.

Furthermore, following Proposition 5.5 in \cite{MR2122282}, in the Jacobi elliptic cone case $0<|A|<1$, the $(y_1,u_2)$ projection has degree greater than one, hence this case is ruled out. The only remaining possibility is that $N$ is a flat 3-plane with multiplicity one. 
\vspace{10pt}
\end{proof}

\subsection{Asymptotics of the special Lagrangians}\label{Asymptotic-behavior}

In this section, we prove $L$ has the expected asymptotic behavior. 

\begin{theorem}
Asymptotically near infinity, 
    $L$ is an exponentially small $C^k$ graph over the model special Lagrangian cylinders $\cup_{i=1}^n L_i$.
\end{theorem}

We divide the proof into a few steps.

\textbf{Step 1 (Cauchy-Riemann type equation).} The region of $L$ close to spatial infinity must project to a small neighbourhood of one of the edges. Morever, Lemma \ref{infinite-wedge} shows that the projection to the $(y_1,y_2)$ plane is close to some ray $\mathcal{R}$, which we can without loss take to be inside $\{ (p_{i+1}- p_i)\cdot y=c_i  \}$. Up to rotating and translating the coordinates, we reduce the problem to the standard position, so the edge lies in $\{ u_1=0 \}$, and the ray is simply $\{ y_2=0 , y_1<0  \}$, and the region lies in  $\{ y_1<0, |y_1|\gg 1\}\cap \{ u_2\in [a,b]\} $, where $[a,b]$ describes the boundary edge.

Via the partial Legendre transform, we see the reduced Lagrangian $L_{red}$ is graphical over the $( u_2, y_1)$ variables, except at the vertices. The special Lagrangian condition can be rewritten as the Cauchy-Riemann type equation
\begin{align*} \label{quasi-elliptic}
    \frac{\partial y_2}{\partial y_1} = -\frac{\partial u_1}{\partial u_2}, \quad \quad \frac{\partial y_2}{\partial u_2} = V \frac{\partial u_1}{\partial y_1}.
\end{align*}
Then Lemma \ref{infinite-wedge} provides the preliminary decay estimate $|y_2|\leq \frac{C}{|y_1|}$. Morever, since $\varphi$ is a convex function with $C^{1/2}$ boundary modulus of continuity, we obtain 
\[
|y_1|u_1= -y_1 u_1= -\partial_{u_1} \varphi(u_1,u_2) u_1 \leq \varphi(0,u_2) -\varphi(u_1,u_2) \leq C u_1^{1/2},
\]
so we have the preliminary estimate $u_1\leq C|y_1|^{-2}$. Thus $L_{red}$ is a $C^0$-small graph over the $(u_2, y_1)$ variables, with decay rate estimate $O(|y_1|^{-1})$.

Now the Cauchy-Riemann type equation is quasi-linear elliptic, so away from the singular locus  $V = \infty$ corresponding to  $u_2\in \{ a,b \}$, we can bootstrap the smallness of the $C^0$-norm to smallness of the $C^k$-norm. More geometrically, the asymptotic model is the half-cylinder $\Sigma_i\times \mathbb{R}^+$. For any given $\epsilon>0$, we can find some $R\gg 1$, such that on the region $\{ y_1<-R, u_2\in [a+\epsilon, b-\epsilon]  \}$ away from the vertex, our special Lagrangian is a $C^k$-small perturbation of the model with $C^1$-norm bounded by $\epsilon$.

\textbf{Step 2 (Quantitative smoothness).} 
We need to prove quantitative smoothness estimate for $L_{red}$ near the vertex region, and in our coordinates this means $u_2$ close to the endpoints $a,b$ of the edges. 
This is based on the Allard's regularity theorem. Notice the ambient manifold has bounded geometry in our region of interest.

\begin{proposition}[Allard's regularity] There exists a universal constant $\epsilon_0 \ll 1$, and some small fixed $r_0$ depending on the ambient manifold such that the following holds. Let $X$ be an $n$-dimensional multiplicity one stationary integral varifold inside the coordinate ball $B(p, r)$ with $r\leq r_0$. Assume that $p$ lies on the support of $X$, and the volume $\mathcal{H}^n(X \cap B(p,r)) \leq (\omega_n+\epsilon_0)r^n$. Then $X \cap B(p,r/2)$ is a $C^{1, \alpha}$ graph over the tangent plane through $p$, with the $C^{1, \alpha}$-norm bounded by $\frac{1}{100}$.
\end{proposition}

Without loss $p$ lies in $a\leq u_2\leq a+\epsilon$, where $\epsilon\ll r_0$ will be fixed later. 
We compute the volume on a small geodesic ball of radius $r\leq r_0$,
\begin{align*}
    \text{Mass}(L_{\text{red}} \cap B(p,r)) &= \int_{L_{\text{red}} \cap B(p,r)} \text{Re}(\Omega) =  2\pi\int_{L_{\text{red}} \cap B(p,r)} (du_1\wedge dy_2-du_2\wedge dy_1).
\end{align*}

On the integration region away from $a\leq u_2\leq a+\epsilon$, by the smallness of the $C^1$-norm of $y_2,u_1$, we see that the integral contribution is bounded by $\omega_3 r^3(1+O(\epsilon))$. On the other hand, the contribution from the region $\{a\leq u_2\leq a+\epsilon\} \cap B(p,r) $ can be estimated by the same idea in Lemma \ref{no-boundary}: the integral of $dy_1\wedge du_2$ is  the Lebesgue area of the $(y_1, u_2)$ projection, which gives a contribution bounded by $O(\epsilon r)$.  The integral of $du_1\wedge dy_2$ is the Lebesgue area of the projection to the $(u_1,y_2)$ plane. Since $u_1,y_2$ are both $O(|y|^{-1})$, this contribution is bounded by $O(|y|^{-2})$. In total,
\[
 \text{Mass}(L_{\text{red}} \cap B(p,r)) \leq \omega_3 r^3(1+C\epsilon) +C\epsilon r+ C|y|^{-2}. 
\]
Now for fixed $r$, we can choose $\epsilon\ll 1$ and $R\gg 1$, such that for $y_1<-R$, all the remainder terms can be dominated by $\epsilon_0 r^3$, so  that
\[
 \text{Mass}(L_{\text{red}} \cap B(p,r)) \leq \omega_3(1+\epsilon_0) r^3.
\]

Thus we can apply Allard regularity to deduce the quantitative smoothness of $L$ close to infinity. Combining with the $C^0$-decay, it follows that $L$ is a $C^k$-small graph over the model half-cylinder.
The $C^k$-norms of $u_1,y_2$  are both bounded by $O(|y|^{-1}).$ In particular, 
\[
\int_{\{ y_1< -1, u_2\in [a,b] \}\cap L  } |\nabla y_2|^2<+\infty.
\]

\textbf{Step 3 (Exponential decay).} It remains to improve the $C^k$ decay to exponential decay.

We first notice that $y_2$ defines a harmonic function on $L$. This is because $y_2$ has zero Hessian on the ambient manifold, and $L$ is a minimal surface.

\begin{lemma}
For any sufficiently large $R\geq R_0$,
    \begin{align}\label{Caccioppoli}
        \int_{ y_1 \leq - R } |\nabla y_2|^2 \leq Ce^{-\gamma R},
    \end{align}
    for a constant $\gamma>0$ independent of $R$.
\end{lemma}

\begin{proof}
We already know that the end of $L$ is a small $C^{k,\alpha}$-graph over the cylindrical model. We define 
\[
F(R) = \int_{ \{y_1\leq -R\} \cap L} |\nabla y_2 |^2 .
\]
Since $y_2$ is harmonic,  for any $c\in \mathbb{R}$, we can apply the divergence theorem to deduce
\[
F(R)= \int_{  L\cap \{ y_1=-R\} } (y_2-c) \nabla y_2 \cdot \vec{n} ,
\]
whence by Cauchy-Schwarz and Poincar\'e inequality,
\[
F(R) \leq  (\int_{ L\cap \{ y_1=-R\}   } |\nabla y_2|^2)^{1/2 } (\int_{ L\cap \{ y_1=-R\}   } |y_2-c|^2)^{1/2 } \leq C\int_{ L\cap \{ y_1=-R\}   } |\nabla y_2|^2.
\]
Thus $-CF' \geq F$, which implies the exponential decay.
\end{proof}

Since $L$ is already $C^k$-regular, we can bootstrap this to $C^k$ exponential decay for $y_2$. Using the elliptic system, it is then easy to see that $L$ is asymptotically an exponentially small graph over the model cylindrical special Lagrangian.

\subsection{Topology of the constructed special Lagrangians}\label{topology}

We conclude by proving the last component of Theorem \ref{generalizationDonaldsonScaduto}, thereby confirming Donaldson-Scaduto Conjecture \ref{Donaldson-Scaduto-conjecture}.

\begin{theorem}
    $L$ is homeomorphic to an $n$-holed 3-sphere.
\end{theorem}

\begin{proof}
Let $L'$ be the 3-manifold obtained by truncating the ends of $L$ at a sufficiently large distance $R$, denoted by $L_R$, and sealing them by adding 3-balls, resulting in a closed 3-manifold. We show $L' \cong S^3$.

\textbf{First Argument} (employing the Poincaré conjecture): We prove that $L'$ is simply connected.

Note that $L$ is fibered over $L_{\text{red}}$, where $L_{\text{red}}$ is homeomorphic to $\overline{\mathbb{D}} \setminus \{a_1, \ldots, a_n\}$ for $n$ distinct boundary points $a_1, \ldots, a_n$. The fiber over any interior point $z \in L_{\text{red}}^{\circ}$ is a copy of $U(1)$, and the fibers collapse to a point when $z \in \partial L_{\text{red}}$. 

Let $C_1, \ldots, C_n$ denote the boundary components of $L_{\text{red}}$. Recalling the projection map $q: u_3^{-1}(0) \to Z_{\text{red}}$, let $V_i := q^{-1}(U_i) \subset L$ be the preimage of an open neighbourhood $U_i$ of the boundary component $C_i$ in $L_{\text{red}}$, such that $U_i \cap U_j = \emptyset$ when $i \neq j$. Each $V_i$ is homeomorphic to $\mathbb{D} \times \mathbb{R}$. Let $U'_i$ be another open neighborhood of the boundary component $C_i$ in $L_{\text{red}}$ slightly smaller than $U_i$. Let $V_0$ be the open set in $L$ defined as the preimage of $U_0 := L_{\text{red}} \setminus \cup_{i=1}^n \overline{U}_i'$. The set $V_0$ is homeomorphic to $\mathbb{D} \times U(1)$. The configuration of open sets in $L_{\text{red}}$ is shown in Figure \ref{fundamental-group}.  

\begin{figure}[H]
\includegraphics[width=6cm]{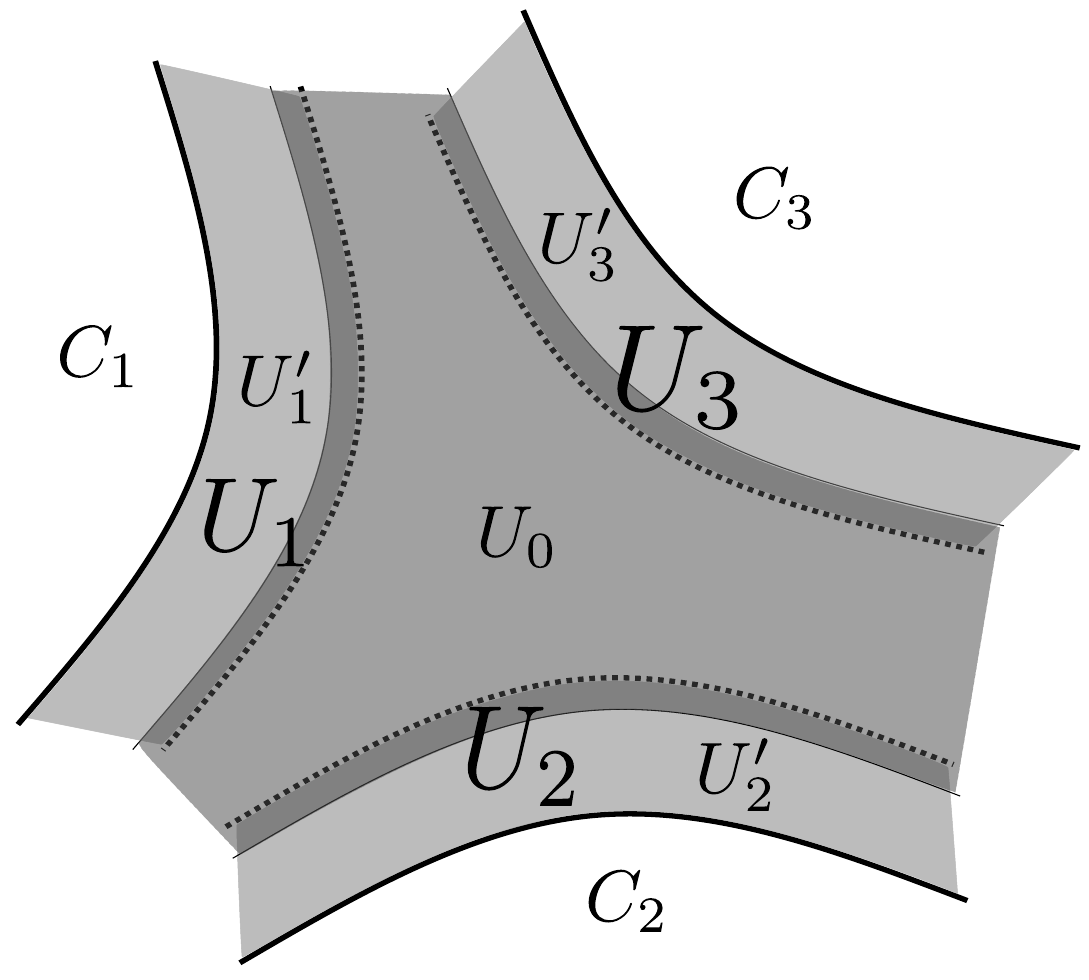}
\centering
  \caption{Boundary components $C_i$ and sets $U_0$ and $U_i$, in the case $n=3$.}
  \label{fundamental-group}
\end{figure}
Let $x_0 \in V_0 \cap V_1$ be the base point. We have $\pi_1(V_0, x_0) \cong \mathbb{Z}$, with a generator presented by a curve encircling the $U(1)$-fiber based at $x_0$. Furthermore, $\pi_1(V_1, x_0) = \{0\}$ and $\pi_1(V_0 \cap V_1, x_0) \cong \mathbb{Z}$, and the inclusion map $ V_0\cap V_1 \to V_0 $  takes the generator of $\pi_1(V_0\cap V_1, x_0)$ to the generator of 
$\pi_1( V_0, x_0)$. Therefore, by Van Kampen's theorem, $\pi_1(V_0 \cup V_1, x_0) = \{0\}$. Applying Van Kampen's theorem again repeatedly and adding $V_i$ inductively yields $\pi_1(V_0 \cup V_1 \cup \ldots \cup V_n, x_0) = \{0\}$, namely $\pi_1(L)=0$. Consequently, $\pi_1(L') = 0$, and therefore, by the Poincaré conjecture, $L'$ is a 3-sphere.

\textbf{Second argument} (without employing the Poincaré conjecture): We can extend the map $q$ to obtain $q: L' \to L_{\text{red}}'$, where $L_{\text{red}}' \cong \bar{\mathbb{D}}^2$ is the truncated version of $L_{\text{red}}$ capped off with $n$ half-discs, where the preimage of any interior point under $q$ is a copy of $S^1$, and the preimage of any boundary point is a single point. In other words, $L_{\text{red}}'$ is an $S^1$-bundle over the interior of $\mathbb{D}^2$, which collapsing to a point above each boundary point.

Let $C$ be an embedded circle in $L_{\text{red}}'$ which divides it into two regions: the interior $D_1$ and the exterior $D_2$. Let $\Sigma = q^{-1}(C)$, and $\Sigma_i = q^{-1}(D_i)$ for $i \in \{1,2\}$. The Heegaard surface $\Sigma = \mathbb{T}^2$, and handlebodies $\Sigma_1 \cong \overline{\mathbb{D}}^2 \times S^1$, and $\Sigma_2 \cong S^1 \times \overline{\mathbb{D}}^2$, leading to a genus 1 Heegaard decomposition of $L'$, where the gluing map of this decomposition maps the meridian of $\Sigma_1$ to the longitude of $\Sigma_2$ and the longitude of $\Sigma_1$ to the meridian of $\Sigma_2$. This description characterizes the genus 1 Heegaard decomposition of $S^3$.

\end{proof}

\section{Appendix: parameter count} \label{appendix}

Our main construction depends on the parameters $\{b_k\}_{k=1}^n$. Changing these parameters by the same additive constant amounts to adding a constant to the Dirichlet solution $\varphi$, which does not affect the special Lagrangian. In total, our construction depends on $n-1$ real parameters. Furthermore, as showed in \cite{MR4495257}, each of these special Lagrangians with fixed asymptotics are rigid.

We now provide an alternative perspective on why the deformations of the $n$ asymptotically cylindrical ends are subject to one additional constraint. We suppose that the $n$ asymptotic half cylinders are contained in $L_i=\Sigma_i \times l_i$, where we recall that $\Sigma_i= \pi^{-1}[p_i,p_{i+1}]$ and $l_i\subset \mathbb{R}^2$ is defined by $l_i = \{y \in \mathbb{R}^2 \; | \;  y\cdot (p_{i+1}-p_i)=c_i\}$ for some $c_i \in \mathbb{R}$.

\begin{lemma}
Let $L$ be an asymptotically cylindrical special Lagrangian in $X\times \mathbb{R}^2$ with asymptotes ends $L_i$ for $i=1,\ldots, n$. Then, we have $\sum_{i=1}^n c_i=0$.
\end{lemma}

\begin{proof}
We can find a primitive for $\text{Im}(\Omega)$,
\begin{align*}
   \text{Im} (\Omega) = \omega_1 \wedge dy_1 + \omega_2 \wedge dy_2 = d \lambda \quad \text{ with } \quad \lambda =  y_1 \omega_1+y_2 \omega_2 .
\end{align*}
We apply the Stokes theorem to $L$ truncated at a very large distance $R$, which is a manifold $L_R$ with boundary diffeomorphic to $\cup_{i=1}^n \Sigma_i$. Since $L$ is a special Lagrangian,
\[
  0 = \int_{L_R}\text{Im}(\Omega) =
   \sum_{i=1}^n  \int_{\Sigma_i} \lambda. 
\]
By definition, $L$ is an exponentially small graph over $L_i$ in the asymptotic regime. Consequently, we can evaluate the boundary integrals on $L_i$ with an error of $O(e^{-cR})$, which disappears in the limit when $R\to +\infty$. We have
\[
(\int_{\Sigma_i} \omega_1, \int_{\Sigma_i}\omega_2)= 2\pi (p_{i+1}-p_i) \in \mathbb{R}^2,
\]
hence 
\[
\sum_{i=1}^n  \int_{\Sigma_i} \lambda=\sum_{i=1}^n \left( y\cdot (p_{i+1}-p_i) +O(e^{-cR}) \right) =\sum_{i=1}^n \left( c_i +O(e^{-cR}) \right),
\]
and letting $R\to +\infty$, we deduce $\sum_1^n c_i=0$.
\end{proof}

\printbibliography

@incollection {MR4479718,
    AUTHOR = {Donaldson, Simon and Scaduto, Christopher},
     TITLE = {Associative submanifolds and gradient cycles},
 BOOKTITLE = {Surveys in differential geometry 2019. {D}ifferential
              geometry, {C}alabi-{Y}au theory, and general relativity.
              {P}art 2},
    SERIES = {Surv. Differ. Geom.},
    VOLUME = {24},
     PAGES = {39--65},
 PUBLISHER = {Int. Press, Boston, MA},
      YEAR = {[2022] \copyright 2022},
      ISBN = {978-1-57146-413-2},
   MRCLASS = {53C26 (53C38)},
  MRNUMBER = {4479718},
MRREVIEWER = {Lorenz\ J.\ Schwachh\"{o}fer},
}

@incollection {MR3702382,
    AUTHOR = {Donaldson, Simon},
     TITLE = {Adiabatic limits of co-associative {K}ovalev-{L}efschetz
              fibrations},
 BOOKTITLE = {Algebra, geometry, and physics in the 21st century},
    SERIES = {Progr. Math.},
    VOLUME = {324},
     PAGES = {1--29},
 PUBLISHER = {Birkh\"{a}user/Springer, Cham},
      YEAR = {2017},
      ISBN = {978-3-319-59938-0; 978-3-319-59939-7},
   MRCLASS = {53C26 (53C38)},
  MRNUMBER = {3702382},
MRREVIEWER = {Antonella\ Nannicini},
       DOI = {10.1007/978-3-319-59939-7\_1},
       URL = {https://doi.org/10.1007/978-3-319-59939-7_1},
}

@book {MR4495257,
    AUTHOR = {Habibi Esfahani, Saman},
     TITLE = {Monopoles, {S}ingularities and {H}yperkahler {G}eometry},
      NOTE = {Thesis (Ph.D.)--State University of New York at Stony Brook},
 PUBLISHER = {ProQuest LLC, Ann Arbor, MI},
      YEAR = {2022},
     PAGES = {205},
      ISBN = {979-8351-45288-3},
   MRCLASS = {Thesis},
  MRNUMBER = {4495257},
       URL =
              {http://gateway.proquest.com/openurl?url_ver=Z39.88-2004&rft_val_fmt=info:ofi/fmt:kev:mtx:dissertation&res_dat=xri:pqm&rft_dat=xri:pqdiss:29327731},
}

@article {MR2122282,
    AUTHOR = {Joyce, Dominic},
     TITLE = {{$\rm U(1)$}-invariant special {L}agrangian 3-folds. {III}.
              {P}roperties of singular solutions},
   JOURNAL = {Adv. Math.},
  FJOURNAL = {Advances in Mathematics},
    VOLUME = {192},
      YEAR = {2005},
    NUMBER = {1},
     PAGES = {135--182},
      ISSN = {0001-8708,1090-2082},
   MRCLASS = {53C38},
  MRNUMBER = {2122282},
       DOI = {10.1016/j.aim.2004.03.016},
       URL = {https://doi.org/10.1016/j.aim.2004.03.016},
}

@article {MR0454331,
    AUTHOR = {Rauch, Jeffrey and Taylor, B. A.},
     TITLE = {The {D}irichlet problem for the multidimensional
              {M}onge-{A}mp\`ere equation},
   JOURNAL = {Rocky Mountain J. Math.},
  FJOURNAL = {The Rocky Mountain Journal of Mathematics},
    VOLUME = {7},
      YEAR = {1977},
    NUMBER = {2},
     PAGES = {345--364},
      ISSN = {0035-7596,1945-3795},
   MRCLASS = {35J40 (32F99)},
  MRNUMBER = {454331},
MRREVIEWER = {B.\ Hellwig},
       DOI = {10.1216/RMJ-1977-7-2-345},
       URL = {https://doi.org/10.1216/RMJ-1977-7-2-345},
}

@article {MR1038359,
    AUTHOR = {Caffarelli, L. A.},
     TITLE = {A localization property of viscosity solutions to the
              {M}onge-{A}mp\`ere equation and their strict convexity},
   JOURNAL = {Ann. of Math. (2)},
  FJOURNAL = {Annals of Mathematics. Second Series},
    VOLUME = {131},
      YEAR = {1990},
    NUMBER = {1},
     PAGES = {129--134},
      ISSN = {0003-486X,1939-8980},
   MRCLASS = {35B65 (35J60)},
  MRNUMBER = {1038359},
MRREVIEWER = {John\ Urbas},
       DOI = {10.2307/1971509},
       URL = {https://doi.org/10.2307/1971509},
}

@article {MR3340380,
    AUTHOR = {Mooney, Connor},
     TITLE = {Partial regularity for singular solutions to the
              {M}onge-{A}mp\`ere equation},
   JOURNAL = {Comm. Pure Appl. Math.},
  FJOURNAL = {Communications on Pure and Applied Mathematics},
    VOLUME = {68},
      YEAR = {2015},
    NUMBER = {6},
     PAGES = {1066--1084},
      ISSN = {0010-3640,1097-0312},
   MRCLASS = {35J96 (35B65 35R45)},
  MRNUMBER = {3340380},
MRREVIEWER = {Nam\ Q.\ Le},
       DOI = {10.1002/cpa.21534},
       URL = {https://doi.org/10.1002/cpa.21534},
}

@article {MR2075484,
    AUTHOR = {Haskins, Mark},
     TITLE = {Special {L}agrangian cones},
   JOURNAL = {Amer. J. Math.},
  FJOURNAL = {American Journal of Mathematics},
    VOLUME = {126},
      YEAR = {2004},
    NUMBER = {4},
     PAGES = {845--871},
      ISSN = {0002-9327,1080-6377},
   MRCLASS = {53C38 (32Q25 53C42)},
  MRNUMBER = {2075484},
MRREVIEWER = {Henri\ Anciaux},
       URL =
              {http://muse.jhu.edu/journals/american_journal_of_mathematics/v126/126.4haskins.pdf},
}

@article {MR4945554,
    AUTHOR = {Bera, Gorapada and Habibi Esfahani, Saman and Li, Yang},
     TITLE = {Uniqueness in the local {D}onaldson-{S}caduto conjecture},
   JOURNAL = {Int. Math. Res. Not. IMRN},
  FJOURNAL = {International Mathematics Research Notices. IMRN},
      YEAR = {2025},
    NUMBER = {16},
     PAGES = {Paper No. rnaf245, 14},
      ISSN = {1073-7928,1687-0247},
   MRCLASS = {53D12 (14J42)},
  MRNUMBER = {4945554},
MRREVIEWER = {Kwang\ Soon\ Park},
       DOI = {10.1093/imrn/rnaf245},
       URL = {https://doi.org/10.1093/imrn/rnaf245},
}

\noindent
\author{Department of Mathematics, Duke University, 120 Science Dr, Durham, NC 27708-0320} \\ E-mail address: \href{ mailto:Saman.HabibiEsfahani@msri.org}{saman.habibiesfahani@duke.edu}

\vspace{10pt}

\noindent
\author{Department of Mathematics, Massachusetts Institute of Technology, 77 Massachusetts Avenue, Cambridge, MA 02139
} \\ E-mail address: \href{ mailto:yangmit@mit.edu}{yangmit@mit.edu} 
\end{document}